\documentclass[reqno]{amsart}  	
\usepackage{geometry}                		
\usepackage{graphicx}				
\usepackage{todonotes}
\usepackage{amssymb}
\usepackage{amsthm}
\usepackage{mathtools}
\usepackage{amsmath,amsfonts}
\usepackage{epsfig}
\usepackage{comment}
\usepackage{color}
\usepackage{soul}
\usepackage{hyperref} 
\hypersetup{
   colorlinks=true,
   linkcolor=blue,   
   citecolor=blue,   
}

\newcommand{\Mod}{{\rm Mod}}

\newcommand{\R}{{\mathbb R}}
\newcommand{\Q}{{\mathbb Q}}
\newcommand{\N}{{\mathbb N}}
\newcommand{\Om}{\Omega}

\newcommand{\cp}{\text{Cap}}

\newcommand{\simge}{\gtrsim}
\newcommand{\simle}{\lesssim}

\newcommand{\rad}{{\rm rad}}

\DeclareMathOperator{\diam}{diam}

\DeclareMathOperator{\dist}{dist}

\DeclareMathOperator{\HB}{{\it HB}}
\DeclareMathOperator{\rcapa}{cap}

\def\vint_#1{\mathchoice
	{\mathop{\vrule width 6pt height 3 pt depth -2.5pt
			\kern -8pt \intop}\nolimits_{#1}}%
	{\mathop{\vrule width 5pt height 3 pt depth -2.6pt
			\kern -6pt \intop}\nolimits_{#1}}%
	{\mathop{\vrule width 5pt height 3 pt depth -2.6pt
			\kern -6pt \intop}\nolimits_{#1}}%
	{\mathop{\vrule width 5pt height 3 pt depth -2.6pt
			\kern -6pt \intop}\nolimits_{#1}}}

\theoremstyle{plain}
\newtheorem{theorem}[equation]{Theorem}

\newtheorem{lemma}[equation]{Lemma}
\newtheorem{proposition}[equation]{Proposition}

\numberwithin{equation}{section}

\theoremstyle{definition}
\newtheorem{definition}[equation]{Definition}

\newtheorem{example}[equation]{Example}
\newtheorem{remark}[equation]{Remark}

\title[Dirichlet problem on unbounded uniform domains and sphericalization]{Solving Dirichlet problem on unbounded uniform domains by using sphericalization techniques}
\author{Riikka Korte, Sari Rogovin, Nageswari Shanmugalingam and Timo Takala}

\begin{document}
	
\begin{abstract}
Within the setting of metric spaces equipped with a doubling measure and supporting a $p$-Poincar\'e inequality, establishing existence of solutions to Dirichlet problem in a bounded domain in such a metric space is accomplished via direct methods of calculus of variation and the use of a Maz'ya type inequality, which is a consequence of the Poincar\'e inequality.
However, when the domain and its boundary are unbounded, such a method is unavailable. In this paper, using the technique of sphericalization developed in the prior paper~\cite{KRST}, we establish the existence of solutions to the Dirichlet boundary value problem for $p$-harmonic functions in unbounded uniform domains with unbounded boundary when $1<p<\infty$. We also explore the issue of whether such solutions are unique by considering $p$-parabolicity and $p$-hyperbolicity properties of the domain.
\end{abstract}

	\maketitle
	
	\vskip .3cm
	
	
	
	\medskip

		
\tableofcontents

\noindent{\bf Acknowledgement:} {\small R.K. is partially supported by the Research Council of Finland through project 360184. N.S. is partially supported by the National Science Foundation (US) grant DMS\#2348748. T.T. was supported by the Finnish Cultural Foundation.}

\section{Introduction} \label{section:intro}

Solving Dirichlet boundary value problems for bounded domains in a metric measure space equipped with a doubling measure and supporting a $p$-Poincar\'e inequality is accomplished by the direct method of calculus of variations, thanks to gaining control over the $L^p$ norms of Sobolev functions in terms of the $L^p$ norms of their gradients. In unbounded domains this is not possible for two reasons; our boundary data belong to the homogeneous Besov classes and so are not necessarily globally integrable, and functions with globally finite energy need not themselves be in the global $L^p$ classes, let alone being controlled in terms of their energy. 
Thus it is beneficial to transform the Dirichlet problem on unbounded domains into a Dirichlet problem on bounded domains via certain conformal transformation of the domain.
In many natural contexts, the unbounded domains of interest are uniform domains, and in~\cite{KRST} we proposed a conformal transformation that preserves the uniformity property of the domain. The transformations considered in~\cite{KRST} belong to the sub-category of transformations known as sphericalizations.
In the present paper we study the corresponding transformation of boundary data and the Dirichlet problem.

The method of transforming an unbounded metric space via sphericalization is a powerful one, first utilized in transforming the complex plane into the bounded space that is the sphere. Sphericalization is taught in most courses on complex analysis, but it has found uses beyond complex analysis. Balogh and Buckley considered a general version of sphericalization in the context of unbounded metric spaces~\cite{BaloghBuckley}. The method of sphericalization was extended to the setting of metric measure spaces in~\cite{BBL, BBL2, EstiXining, EstiXining2, GibaraKorteSh, Xining1, XiningNages} by including transformations of the measures as well. Apart from~\cite{GibaraKorteSh}, these works constructed the sphericalization via a chaining procedure, which forces there to be only one new boundary point at infinity even when the domain in question might have more than one end at infinity. The construction of the sphericalized metric in~\cite{GibaraKorteSh,GibaraSh,KRST} is done via path integrals, and hence if the domain of interest has more than one end at infinity, then the resulting transformed space will have the corresponding number of new boundary points. In the present paper, as in~\cite{GibaraKorteSh,GibaraSh,KRST}, we are concerned with an unbounded uniform domain, and so there is only one end at infinity.

The Dirichlet boundary value problem considered in this paper corresponds to the classical situation in a Euclidean domain $\Om$, of finding a Sobolev function $u$ such that 
\[
-\Delta_p u=0 \text{ on }\Om,
\] 
and the trace of $u$ to the boundary $\partial\Om$, denoted
$Tu$, satisfies 
\[
Tu=f \text{ on }\partial\Om
\] 
for a given Dirichlet data $f$ on $\partial\Om$. Interested readers can find information regarding this problem from~\cite{HKM} and the references therein. 
The analog of the equation $-\Delta_pu=0$ in the metric setting corresponds to minimization of $p$-energy in the domain; a good background in the metric setting can be found in~\cite{bjornbjorn} and the references given there.

To achieve the goal of solving Dirichlet boundary value problems on unbounded uniform domains, we considered a wider class of sphericalization transformations in the context of unbounded uniform domains in the paper~\cite{KRST}, and the tools developed there are used in the present paper to incorporate the transformation of the boundary of the uniform domains. We point out here that the method of transforming an unbounded uniform domain, its measure, and its boundary into a bounded uniform domain equipped with the transformed measure and the transformed boundary is useful beyond the goal of the present paper. For instance, exploring the large-scale behavior of certain functions in the unbounded domain is made easier by considering the corresponding behavior of the transformed function near the new boundary point of  the transformed domain that corresponds to infinity.

{\bf Structural assumptions.}
Throughout this paper, we fix a parameter $1\le p<\infty$ and assume the following:
\begin{enumerate}
    \item $(X,d,\mu)$ is an unbounded complete metric measure space with a doubling measure $\mu$ and supporting a $p$-Poincar\'e inequality. 
    \item $\Omega\subset X$ is an unbounded uniform domain, whose boundary $\partial \Omega$ is unbounded, uniformly perfect with constant $\kappa$, and equipped with a Radon measure $\nu$ that is $\theta$-codimensional with respect to $\mu$ with some $0<\theta<p$ as given in~\eqref{eq:codim-def}.
\end{enumerate}
Uniformity of the domain $\Omega$ together with the doubling property of $\mu$ on $(X,d)$ and the support of a $p$-Poincar\'e inequality on $(X,d,\mu)$ guarantees that $\mu$ is a doubling measure on $(\Omega,d)$ and $(\Omega,d,\mu)$ supports a $p$-Poincar\'e inequality, see \cite[Theorem 4.4]{bjornshanmugalingam}.
Thus when we study boundary value problems on a uniform domain $\Omega$, the ambient metric measure space $X$ does not play any role, so we take $X=\overline\Omega^d$, where $\overline\Omega^d$ denotes the metric completion of $\Omega$. The boundary of $\Omega$ is then $\partial \Omega=\overline\Omega^d\setminus \Omega$.

In the current paper the sphericalization is defined by using a positive metric density function~$\rho$ that is defined radially around a fixed base point $b \in \partial \Om$.
The new metric $d_{\rho}$ is defined via path integrals of $\rho$. The new measure $\mu_{\rho}$ is defined continuously with respect to $\mu$ by using $\rho^p$ as a weight function. Similarly the new measure $\nu_{\rho}$ is defined continuously with respect to $\nu$ by using $\rho^{p-\theta}$ as a weight function. See Section~\ref{sect-spher-def} for the precise definitions. 

We assume that the metric density function $\rho : (0,\infty) \to (0,\infty)$ is lower semicontinuous and satisfies the following conditions, which we call conditions \ref{condA}, \ref{condB}, \ref{condC}, and \ref{condD}. Here and throughout this paper we set $|x| := d(x,b)$.
\begin{enumerate}
\renewcommand{\labelenumi}{\textbf{\theenumi}}
\renewcommand{\theenumi}{(\Alph{enumi})}
\item\label{condA}
There exists a positive constant $C_A$ such that whenever $0 < r \leq 2s+1$ and $0 < s \leq 2r+1$,
\begin{equation*} 
\rho(r) \leq C_A \rho(s).
\end{equation*}
\item\label{condB}
There exists a positive constant $C_B$ such that for every $r > 0$,
\begin{equation*}
\int_r^{\infty} \rho(t) dt
\leq C_B (r+1) \rho(r).
\end{equation*}
\item\label{condC}
There exists a positive constant $C_C$ such that for every $r > 0$,
\begin{equation*} 
\int_{\Omega\setminus B(b,r)} \rho(|x|)^p d\mu(x)
\le C_C \rho(r)^p \mu(B(b,r+1)).
\end{equation*}
\item
\label{condD}
There exists a positive constant $C_D$ such that for every $r > 0$,
\begin{equation*} 
\int_{\partial \Omega \setminus B(b,r)} \rho(|x|)^{p-\theta} d\nu(x)
\leq C_D \rho(r)^{p-\theta} \nu(B(b,r+1)).
\end{equation*}
\end{enumerate}
The conditions~\ref{condA}, \ref{condB}, and~\ref{condC} are the same as in \cite{KRST} with the choice of $\sigma = p$.
Condition~\ref{condD} is an additional condition that is required to preserve the codimension property. 
The codimensionality relationship between $\nu_\rho$ and $\mu_{\rho}$, proved in Section~\ref{sect-codim}, requires $\nu_\rho$ to be doubling as well, and condition~\ref{condD} is needed in establishing the doubling property of $\nu_\rho$ (mirroring the requirement of condition~\ref{condC} in establishing the doubling property of $\mu_\rho$, see~\cite{KRST}). In light of condition~\ref{condA} and the doubling properties of $\mu$ and $\nu$, the conditions~\ref{condB}, \ref{condC}, and~\ref{condD} need to be verified for a given function $\rho$ only for the case $r \ge 1$.
Note that the doubling property of $\nu$ follows from the doubling property of $\mu$ and the codimensionality, see also the discussion after Definition \ref{codim-def}.

The first main theorem of this paper is the following, establishing existence of a solution to the Dirichlet problem on unbounded uniform domains.

\begin{theorem}\label{thm:Main}
Let $p > 1$ and suppose that $(\Om, d,\mu)$ satisfies the structural assumptions listed above.
For each $f\in HB^{1-\theta/p}_{p,p}(\partial\Om,d,\nu)$ there is a unique function $u_f\in D^{1,p}(\Om,d,\mu)$ that is $p$-harmonic in $(\Om,d,\mu)$ and $T u_f = f$ $\nu$-a.e.~on $\partial\Om$.
\end{theorem}

The concrete definition of the notion of $p$-harmonicity, referred to in the above theorem, is given in Definition~\ref{def:DProb} below.
The function space $HB^{1-\theta/p}_{p,p}(\partial\Om,d,\nu)$, called the homogeneous Besov space of smoothness index $1-\theta/p$, is the trace-space of the homogeneous Sobolev space $D^{1,p}(\Om,d,\mu)$, and its definition is given in Definition~\ref{def:Besov} below.

The proof of the theorem, undertaken in Section~\ref{Sec:MainProof}, uses the tools of sphericalization corresponding to $\rho$ satisfying the above four conditions, and the sections prior to Section~\ref{Sec:MainProof} are devoted to developing the properties of the transformation. 
In~\cite{KRST} it was established that if the function $\rho$ satisfies conditions~\ref{condA} and \ref{condB}, then the property of being a uniform domain is preserved under the change of the metric on $\Om$ from $d$ to $d_\rho$ (see Section~\ref{sect-spher-def} for the construction of $d_\rho$). Furthermore, it was shown there that when $(\Om,d)$ is an unbounded uniform domain, the boundary $\partial_\rho\Om:=\overline{\Om}^{d_\rho}\setminus\Om$ has only one more point than $\partial\Om=\overline{\Om}^d\setminus \Om$.
Moreover, it was shown there that the corresponding transformation $\mu_\rho$ of the doubling measure $\mu$ on $(\Om, d)$ satisfies the doubling property on $(\Om,d_\rho)$, if $\rho$ also satisfies condition~\ref{condC}.

However, to establish Theorem~\ref{thm:Main}, we also need to consider how trace operators, acting on homogeneous Sobolev spaces (see Definition~\ref{def:Sobolev}) to give functions on the boundary $\partial\Om$, are transformed by the sphericalization via $\rho$.
The sections of this paper prior to Section~\ref{sect-Besov} develop the tools necessary to understand how the measure $\nu$ on $\partial\Om$ (that is $\theta$-codimensional with respect to $\mu$), transformed by sphericalization to a measure $\nu_\rho$, retains its doubling property and $\theta$-codimensionality with respect to the sphericalized measure $\mu_\rho$.
These preservations are established in Theorems~\ref{thm:nu-rho-doubling} and \ref{thm:nu-rho-codim}. Subsequently, in Theorem~\ref{thm:preserve-Besov} we establish that the homogeneous Besov class of functions on $(\partial\Om,d,\nu)$ coincides with the Besov class on $(\partial_\rho \Omega,d_\rho,\nu_\rho)$ and that the two Besov energies are comparable (i.e., quasipreserved under the sphericalization transformation). These Besov classes are the traces of the respective Sobolev classes on $\Om$, and contain functions that serve as Dirichlet boundary data for the Dirichlet boundary value problem referenced in Theorem~\ref{thm:Main}. The formal definition of the Dirichlet boundary value problem referred to in Theorem~\ref{thm:Main} is discussed in Section~\ref{sect-dirichlet-notions}, and the subsequent Section~\ref{Sec:MainProof} is devoted to the proof of Theorem~\ref{thm:Main}.

Since $\partial_\rho\Om$ has one more point, $\infty$, than $\partial\Om$, it is of valid concern to know what effect this extra point has on the solutions. The notion of $p$-hyperbolicity of a metric measure space corresponds to the space having positive capacity at large scales, that is, ``there is a lot of room at $\infty$''. When $p=2$, this notion, in the setting of Riemannian manifolds, corresponds to transience of Brownian motion, that is, almost surely a Brownian path eventually leaves every compact set and tends to $\infty$.
The $p$-hyperbolicity of $(\overline \Om^d,d,\mu)$ is equivalent to the $p$-capacity of the singleton set $\{\infty\}\subset\partial_\rho\Om$ being positive, see Proposition~\ref{lem-parabolic-capacity}.
From the results of~\cite{GibaraKorteSh}, we know that when $\partial\Om$ is bounded and $\overline{\Om}^d$ is $p$-hyperbolic, there are infinitely many possible solutions to the Dirichlet problem as given in Definition~\ref{def:DProb}. However, as we show in Theorem~\ref{thm:Main}, this is not the case when $\partial\Om$ is unbounded, and instead we always have a unique solution. The reason for this phenomenon comes from Theorem~\ref{lem-zeroboundaryvalues}, which shows that transformation of functions in the homogeneous Sobolev class $D_0^{1,p}(\Om,d,\mu)$ have the uniquely determined trace value zero at $\infty$ when $\overline{\Om}^d$ is $p$-hyperbolic.
When $\overline{\Om}^d$ is $p$-parabolic, the trace value of these functions at $\infty$ does not matter as $\{\infty\}$ will then have zero $p$-capacity in ($\overline{\Om}^{d_\rho},d_\rho,\mu_\rho)$.

We contrast this with the works of Hansevi~\cite{Hansevi2, Hansevi1}. From the results in~\cite{Hansevi1} we know that there is a unique global energy minimizing solution to the Dirichlet boundary value problem in $\Om$ provided that the $p$-capacity of the complement of $\Om$ is positive; however, the notion of energy minimization considered in~\cite{Hansevi1} requires that the solution should have minimal energy amongst the class of \emph{all} functions in $D^{1,p}(\Om)$ with the prescribed boundary data as the trace, not just functions that are perturbations, merely on compact subsets of the domain, of the minimizer.
This distinction leads to a multitude of solutions when $\partial\Om$ is bounded and $\overline \Omega^d$ is $p$-hyperbolic~\cite{GibaraKorteSh}, while~\cite{Hansevi1} obtains only one solution. What is novel in Theorem~\ref{thm:Main} is that \emph{even when $\overline \Omega^d$ is $p$-hyperbolic}, the solutions obtained via comparison with compactly supported perturbations on $\Om$ are unique when $\partial\Om$ is unbounded. We emphasize here that what we consider as solutions gives a broader class than~\cite{Hansevi1}.
In~\cite{Hansevi2}, our notion of solutions is considered and constructed via Perron method provided $\overline \Omega^d$ is $p$-parabolic.

Given the surprising phenomenon regarding $p$-hyperbolicity pointed out in the above paragraph, we wish to explore the geometric properties that guarantee $p$-hyperbolicity and those that guarantee $p$-parabolicity. In this paper, we explore the notions of $p$-parabolicity and $p$-hyperbolicity of $\overline \Omega^d$ in terms of volume growth; the following is the second main theorem of the present paper.

\begin{theorem}\label{thm:Main2}
Let $p>1$ and assume that $(\Om,d,\mu)$ satisfies the structural assumptions listed at the beginning of this section. 
Then $(\overline{\Om}^d,d,\mu)$ is $p$-parabolic if
\[
\liminf_{R\to\infty}\, \frac{\mu(B(b,R))}{R^p}<\infty,
\]
and $(\overline{\Om}^d,d,\mu)$ is $p$-hyperbolic if there is some $q_0>p$ such that 
\[
\liminf_{R\to\infty}\, \frac{\mu(B(b,R))}{R^{q_0}}\, >0.
\]
\end{theorem}


\begin{remark}
The conclusion of the above theorem holds even if $\partial\Om$ is bounded, as can be seen by following its proof.
\end{remark}

We prove Theorem~\ref{thm:Main2} in Section~\ref{sect-hyperbolicity} by using the sphericalization procedure developed in this paper before Section \ref{sect-Besov} and in~\cite{KRST}.
The first part of Theorem \ref{thm:Main2} is proved as Theorem~\ref{thm:p-parabolic}, and the second claim is proved as Theorem~\ref{thm:hypo}. They key result in Section~\ref{sect-hyperbolicity} that allows us to use the tools of sphericalization is Proposition~\ref{lem-parabolic-capacity}; this proposition allows us to link the $p$-modulus of families of curves in $\overline{\Om}^d$ that leave every compact set in $\overline{\Om}^d$ to the $p$-capacity of $\{\infty\}$ with respect to the sphericalized space $(\overline{\Om}^{d_\rho}, d_\rho, \mu_\rho)$.
Results regarding related volume-growth phenomena in the context of manifolds can be found in~\cite{HoloKos}.

In Section~\ref{Sec:Examples} we also consider concrete examples of metric density functions $\rho$, one of which was considered earlier in~\cite{bjornbjornkorterogovintakala}, see Example~\ref{Example-2}.

\section{Definitions and preliminary results}
\label{sect-def-prelim}

In this section we let $(Y,d_Y)$ be a metric space.
For $x\in Y$ and $r>0$, the (open) ball centered at $x$ with radius $r$ is denoted $B(x,r) := \{ y \in Y : d_Y(x,y) < r \}$.
Given a ball $B=B(x,r)$ and a positive number $\kappa$, the set $\kappa B$ denotes the concentric ball $B(x,\kappa r)$.

\begin{definition}
Let $(Y,d_Y)$ be a metric space.
For any set $A\subset Y$, we say that $A$ is \emph{uniformly perfect at point} $a\in A$ if there exists a constant $\kappa > 1$ such that $(B(a,\kappa r)\setminus B(a,r))\cap A\neq \varnothing$ for every $r>0$ such that $A \cap B(a,r) \neq A$. 
A set $A$ is \emph{uniformly perfect} if $A$ is uniformly perfect at every point $a\in A$ with the same constant $\kappa$.
\end{definition}

\begin{definition}
A measure $\mu_Y$ on $(Y,d_Y)$ is said to be \emph{doubling}, if $\mu_Y$ is a Borel measure and there is a constant $C_\mu\ge 1$ such that
\[
0<\mu_Y(B(x,2r))\le C_\mu\mu_Y(B(x,r))<\infty
\]
for every $x\in Y$ and $r>0$.
\end{definition}

If the measure $\mu_Y$ is doubling and the space is uniformly perfect at $x$, then there exist $0<\alpha_\mu\leq\beta_\mu<\infty$ such that for every $0<r \leq R<\diam(Y)$ we have
\begin{equation}\label{eqn:massbounds}
\left(\frac{r}{R}\right)^{\beta_\mu} \simle \frac{\mu_Y(B(x,r))}{\mu_Y(B(x,R))}\simle \left(\frac{r}{R}\right)^{\alpha_\mu}.
\end{equation}
The upper mass bound exponent $\alpha_\mu$ depends on both the doubling constant $C_\mu$ and the uniform perfectness constant $\kappa$, see \cite[Lemma 2.2]{bjornbjornkorterogovintakala}. It is easy to see that we can always choose $\beta_\mu=\log_2 C_\mu$.
In \eqref{eqn:massbounds} the comparison constants depend only on $C_\mu$ and $\kappa$.

\begin{definition}
A domain $\Omega$ in a complete metric space $(Y,d_Y)$ is said to be \emph{uniform} if $\Om$ is locally compact but not complete, and there exists a constant $C_U \geq 1$ such that for every pair of points $x,y\in\Omega$ there exists a rectifiable curve $\gamma \subset \Om$ joining the points with the property that $\ell(\gamma)\le C_U d_Y(x,y)$ and
\[
\min\{\ell(\gamma_{[x,z]}),\ell(\gamma_{[z,y]})\}\leq C_U \dist(z,Y\setminus \Omega)
\]
for every $z \in \gamma$. Here $\gamma_{[x,z]}$ denotes any subarc of $\gamma$ between $x$ and $z$, and $\ell(\gamma)$ denotes the length of $\gamma$.
\end{definition}

If there exists a doubling measure $\mu_Y$ in the complete metric space $(Y,d_Y)$, then $(Y,d_Y)$ is also proper (that is, closed and bounded subsets of $(Y,d_Y)$ are compact), and so every domain in $(Y,d_Y)$ is locally compact, see \cite[Proposition 3.1]{bjornbjorn}.

\begin{definition}
\label{codim-def}
Let $\Om$ be an open set in the complete metric space $(Y, d_Y)$ such that $\partial \Omega \neq \varnothing$. Given a Radon measure $\mu_Y$ in $Y$, a Radon measure $\nu_Y$ on $\partial\Omega$, and $\theta > 0$, we say that $\nu_Y$ is \emph{$\theta$-codimensional} with respect to $\mu_Y$ if there exists a constant $C_\theta\ge 1$ such that
\begin{equation}\label{eq:codim-def}
\frac{1}{C_\theta}\frac{\mu_Y(B(x,r))}{r^\theta}\le \nu_Y(B(x,r))\le C_\theta\frac{\mu_Y(B(x,r))}{r^\theta}
\end{equation}
for every $x\in\partial\Omega$ and $0<r<2\diam(\partial\Om)$.
\end{definition}

If $\mu_Y$ is doubling and $\nu_Y$ is $\theta$-codimensional with respect to $\mu_Y$, it follows that $\nu_Y$ is also doubling with a doubling constant $C_{\nu}=2^{-\theta}\,C_\mu C_\theta^2$.
Also if $\mu_Y$ is doubling, since the codimensional measure $\nu_Y$ is supported on $\partial\Om$, necessarily $\mu_Y(\partial\Om)=0$, see for instance~\cite[Lemma~8.1]{GibaraKorteSh}.

Let $\mu_Y$ be a Borel measure in the space $(Y,d_Y)$.
For $\theta>0$ and $0<R<\infty$, let $\mathcal H_R^{-\theta}$ denote the Hausdorff-content of codimension $\theta$ at scale $R$, i.e.~for a set $A\subset Y$ we set
\begin{equation}\label{eq:def-codim-Haus}
\mathcal H_R^{-\theta}(A) :=
\inf\left\{\sum_{i\in I\subset \mathbb N} \frac{\mu_Y(B(x_i,r_i))}{r_i^\theta}\,:\, A\subset \bigcup_{i\in I} B(x_i,r_i) 
\text{ and } r_i\le R \text{ for all } i \in I \right\}.
\end{equation}
The Hausdorff-measure of codimension $\theta$ is $\mathcal H^{-\theta}(A) := \lim_{R\rightarrow 0^+} \mathcal H_R^{-\theta}(A)$.
In~\eqref{eq:def-codim-Haus}, we can impose the additional condition that $B(x_i,r_i)\cap A$ is non-empty for each $i\in I$, as discarding balls in the cover that do not intersect $A$ just results in a cover with a smaller sum.

\begin{lemma}\label{lemma:measure content}
Suppose that $(Y,d_Y,\mu_Y)$ is a metric measure space with $\mu_Y$ a Borel measure, $x_0\in Y$, $r>0$, and suppose that $E \subset B(x_0,r)$ such that $\mu_Y(E) \geq \alpha \mu_Y(B(x_0,r))$ for some $0 < \alpha \leq 1$. Then
\[
\alpha \frac{\mu_Y(B(x_0,r))}{r^\theta}
\leq \mathcal{H}_r^{-\theta}(E)
\leq \frac{\mu_Y(B(x_0,r))}{r^\theta}.
\]
\end{lemma}

\begin{proof}
Since $\{B(x_0,r)\}$ is a cover of $E$, it follows that $\mathcal{H}^{-\theta}_r(E)\le \mu_Y(B(x_0,r))/r^\theta$.
On the other hand, if $\{B_i\, :\, i\in I\subset\N\}$ is a countable cover of $E$ such that for each $i\in I$ we have $\rad(B_i) \leq r$, then
\[
\sum_{i\in I}\frac{\mu_Y(B_i)}{\rad(B_i)^\theta}
\ge \frac{1}{r^\theta}\, \sum_{i\in I}\mu_Y(B_i)
\ge \frac{1}{r^\theta}\, \mu_Y(E)
\geq \alpha \frac{\mu_Y(B(x_0,r))}{r^\theta}.
\qedhere
\]
\end{proof}

\begin{lemma}\label{lem:Karadharai}
Suppose that $(Y,d_Y,\mu_Y)$ is a metric measure space with $\mu_Y$ a doubling measure. Let $R>0$, $\theta > 0$, $\tau\ge 1$, and $E\subset Y$.
Then 
\[
\mathcal{H}^{-\theta}_{\tau R}(E)
\le 
\mathcal{H}^{-\theta}_{R}(E)
\simle \mathcal{H}^{-\theta}_{\tau R}(E),
\]
with the comparison constant depending solely on the doubling constant $C_\mu$ of $\mu_Y$, $\tau$, and $\theta$.
\end{lemma}

\begin{proof}

The left-most inequality follows from the definition of $\mathcal{H}^{-\theta}_r$, so it suffices to prove the right-most inequality. It is enough to prove the right-most inequality for the choice of $\tau=2$, as that already implies the result for the case $\tau \leq 2$ and an iteration of the inequality would then give the corresponding result for all $\tau>2$ as well.

To this end, let $\{B_i\, :\, i\in I\}$ with $I\subset \mathbb{N}$ be a countable cover of $E$ by balls of radii at most $2R$. For each $i\in I$, we have two possibilities; either, $\rad(B_i)\le R$, or $R<\rad(B_i)\le 2R$. In the first case we set $\mathcal{F}_i=\{B_i\}$, but in the second case, we can find at most $N$ number of balls, with center in $B_i$ and radius $R$, covering $B_i$, and we denote this collection by $\mathcal{F}_i$. Note that here $N$ depends solely on the doubling constant $C_\mu$ of $\mu_Y$, see for instance~\cite{Hei, HKST}. We now get a cover $\bigcup_{i\in I}\mathcal{F}_i$ of $E$ by balls of radius at most $R$, and so, with $I_1$ the collection of all indices $i\in I$ for which $\rad(B_i)\le R$,
\[
\mathcal{H}^{-\theta}_R(E)\le  \sum_{i\in I_1}\frac{\mu_Y(B_i)}{\rad(B_i)^\theta}
+\sum_{i\in I\setminus I_1}\frac{1}{R^\theta}\, \sum_{B\in \mathcal{F}_i}\mu_Y(B)
\le \sum_{i\in I_1}\frac{\mu_Y(B_i)}{\rad(B_i)^\theta}
+\sum_{i\in I\setminus I_1}\frac{1}{R^\theta}\, N\, C_\mu\, \mu_Y(B_i).
\]
As $\rad(B_i)\le 2R$ for each $i\in I\setminus I_1$, it follows that
\[
\mathcal{H}^{-\theta}_R(E)\le  \sum_{i\in I_1}\frac{\mu_Y(B_i)}{\rad(B_i)^\theta} 
+2^\theta\, N\, C_\mu\, \sum_{i\in I\setminus I_1}\frac{\mu_Y(B_i)}{\rad(B_i)^\theta}
\le 2^\theta \, N\, C_\mu\, \sum_{i\in I} \frac{\mu_Y(B_i)}{\rad(B_i)^\theta}.
\]
Taking the infimum over all such covers yields the right-most inequality.
\end{proof}

\begin{definition}
Let $p \geq 1$ and $\mu_Y$ a Borel measure on the metric space $(Y,d_Y)$. Let $\Gamma$ be a family of curves on $Y$. Then we define the \emph{$p$-modulus} of $\Gamma$ by
\begin{equation*}
\Mod_p(\Gamma)
= \inf \int_Y g^p d \mu_Y,
\end{equation*}
where the infimum is over all nonnegative Borel functions $g$ such that $\int_{\gamma} g ds \geq 1$ for every locally rectifiable curve $\gamma \in \Gamma$.
\end{definition}

See \cite[Section 5.2]{HKST} for more on the modulus of curve families.

Following Heinonen--Koskela~\cite{HeinonenKoskelaActa}, given a function $u:Y\rightarrow [-\infty,\infty]$, we say that a nonnegative Borel function $g$ on $Y$ is an \emph{upper gradient} of $u$ if
\[
|u(x)-u(y)|\leq\int_\gamma g\, ds
\]
for every pair of points $x,y\in Y$ and for every non-constant rectifiable curve $\gamma$ connecting $x$ to $y$.
When at least one of $u(x)$ and $u(y)$ is not a real number, we interpret the above inequality to also mean that $\int_\gamma g\, ds$ is infinite.
See~\cite[Section 5]{HKST} for more information about path integrals.
On curves $\gamma$ where $\int_\gamma g\, ds$ is finite, such an upper gradient controls the magnitude of the derivative of the absolutely continuous function $u\circ\gamma$.
A fruitful theory of Sobolev-type function spaces, potential theory, and theory of quasiconformal mappings in the nonsmooth metric setting has been developed during the past two decades, see for instance~\cite{bjornbjorn,Hei,HeinonenKoskelaActa,HKST,Keith1} and the references therein.
Further developments using related notions stemming from the upper gradient notion can be found in~\cite{AmbrosioGigliSavare,Chee,Keith2} and the references therein.
Let $\mu_Y$ be a Borel measure on $(Y,d_Y)$.
If the function $g$ satisfies the previous inequality for all curves except for a curve family that has $p$-modulus zero, then $g$ is a \emph{$p$-weak upper gradient} of $u$.
When a function $u$ has a $p$-weak upper gradient $g$ with $\int_Y g^p\, d\mu_Y$ finite, there is a unique $p$-weak upper gradient $g_u$ of $u$ with the smallest $L^p$ norm amongst the $L^p$ norms of all $p$-weak upper gradients $g$ of $u$. Following~\cite{HKST}, we say that $g_u$ is the \emph{minimal $p$-weak upper gradient} of $u$.

\begin{definition}\label{def:Sobolev}
A measurable function $u$ on $Y$ is said to be in the \emph{homogeneous Sobolev space} $D^{1,p}(Y,d_Y,\mu_Y)$ if $u$ has a $p$-weak upper gradient $g\in L^p(Y,\mu_Y)$.
The \emph{Newton-Sobolev space} $N^{1,p}(Y,d_Y,\mu_Y)$ consists of functions $u$ that are in $L^p(Y,\mu_Y)$ and $D^{1,p}(Y,d_Y,\mu_Y)$.
The norm for a function $u$ in the space $D^{1,p}(Y,d_Y,\mu_Y)$ is defined by $\|u\|_{D^{1,p}(Y, d_Y,\mu_Y)}:=\|g_u\|_{L^p(Y,\mu_Y)}$. 
The norm for a function $u$ in the space $N^{1,p}(Y,d_Y,\mu_Y)$ is defined by $\|u\|_{N^{1,p}(Y, d_Y,\mu_Y)}:=\|g_u\|_{L^p(Y,\mu_Y)} + \|u\|_{L^p(Y,\mu_Y)}$.
\end{definition}

In establishing existence of solutions to certain Dirichlet problems, the notion of $p$-capacity plays a key role.
Associated with the notion of capacity is the notion of variational capacity of a condenser.

\begin{definition}\label{def:capacities}
Given a set $E\subset Y$, the $p$-\emph{capacity} $\cp_p(E)$ is the number
\[
\cp_p(E):=\inf_u\int_Y[|u|^p+g_u^p]\, d\mu_Y,
\]
where the infimum is over all functions $u\in N^{1,p}(Y, d_Y,\mu_Y)$ satisfying $u\ge 1$ on $E$. Because minimal $p$-weak upper gradients satisfy a truncation property, we may equivalently consider the infimum to be over all $u\in N^{1,p}(Y, d_Y,\mu_Y)$ with $0\le u\le 1$ on $Y$ and $u=1$ on $E$.

Given an open set $U\subset Y$ and two disjoint sets $E,F\subset U$, the \emph{variational $p$-capacity} of the condenser $(E,F;U)$ is the number
\[
\rcapa_p(E,F; U):=\inf_u\int_U g_u^p\, d\mu_Y,
\]
where the infimum is over all functions $u\in N^{1,p}(U,d_Y,\mu_Y\vert_U)$ with $u\ge 1$ on $E$ and $u\le 0$ on $F$.
\end{definition}

The property of $\cp_p(E)=0$ is equivalent to the property that $\mu_Y(E)=0$ together with $\Mod_p(\Gamma_E)=0$, where $\Gamma_E$ is the collection of all non-constant compact rectifiable curves in $Y$ that intersect $E$.
We refer the interested reader to~\cite{bjornbjorn,HKST} for more on the notions of $p$-capacity and variational $p$-capacity.

For any $\mu_Y$-measurable set $A\subset Y$ with $0<\mu_Y(A)<\infty$ and any $u \in L^1(A)$, we denote the integral average of $u$ on $A$ by $u_A := \vint_A u\, d\mu_Y := \tfrac{1}{\mu_Y(A)}\int_A u\, d\mu_Y$.

\begin{definition}
The metric measure space $(Y,d_Y,\mu_Y)$ supports a \emph{$p$-Poincar\'e inequality} with $1 \le p<\infty$ if $\mu_Y$ is a Borel measure and there are constants $C_P> 0$ and $\lambda\ge1$ such that for every ball $B=B(x,r)\subset Y$, for every measurable function $u$, and for every upper gradient $g$ of $u$, we have 
\[
\vint_{B(x,r)}|u-u_B| d\mu_Y \le C_P r \left(\vint_{B(x,\lambda r) }g^p d\mu_Y\right)^{1/p}.
\]
\end{definition}
If $Y$ is a geodesic space and $\mu_Y$ is a doubling measure, then we can take $\lambda=1$ at the expense of increasing the constant $C_P$, see \cite[Theorem 9.1.15]{HKST}.

The support of $p$-Poincar\'e inequality implicitly implies that for each ball $B\subset Y$ we have that every function $u\in D^{1,p}(Y,d_Y,\mu_Y)$ satisfies $\int_B|u|\, d\mu_Y<\infty$. This latter condition will be denoted by saying that $u\in L^1(B,\mu_Y)$, but the reader should keep in mind that functions in $D^{1,p}(Y,d_Y,\mu_Y)$ cannot be perturbed arbitrarily on sets of $\mu_Y$-measure zero (though they can be perturbed arbitrarily on sets of $p$-capacity zero).

The following proposition might be of independent interest, and it is needed for the proof of Theorem~\ref{lem-zeroboundaryvalues} below.
In~\cite{HeinonenKoskelaActa}, Heinonen and Koskela considered (locally) Ahlfors $Q$-regular complete metric measure spaces with $Q>1$, and in that setting
they proved that \eqref{capacityequation}, with $p=Q$, is
equivalent to the space supporting a $Q$-Poincar\'e inequality.
In that restricted setting of Ahlfors $Q$-regularity with $p=Q$, the property is called the Loewner property in~\cite{HeinonenKoskelaActa}, so the following proposition is a generalization of the Loewner property.
We point out that the assumption that the measure $\mu_Y$ is doubling is weaker than the assumption of Ahlfors regularity, and we are free to consider any choice of $p$ provided $1\le p<\infty$.

\begin{proposition}\label{prop:capest}
Let $(Y,d_Y,\mu_Y)$ be a metric measure space with a doubling measure and supporting a $p$-Poincar\'e inequality. Let $E$ and $F$ be two disjoint subsets of the ball $B_R=B(x_0,R) $ in $Y$.  
Assume that for some $0<\theta<p$ and $M>0$ we have 
\[
\min\{\mathcal H_R^{-\theta}(E),\mathcal H_R^{-\theta}(F)\}\ge M \frac{\mu_Y(B_R)}{R^\theta}.
\] 
Then
\begin{equation}
\label{capacityequation}
\rcapa_p(E,F;2\lambda B_R)\simge M \frac{\mu_Y(B_R)}{R^p}
\end{equation}
with the comparison constant depending only on $C_{\mu}$, $p$, $\lambda$, $C_P$ and $\theta$.
\end{proposition}

From the definition of $\mathcal{H}_R^{-\theta}$ given in~\eqref{eq:def-codim-Haus}, it is clear that $\mathcal{H}^{-\theta}\ge \mathcal{H}^{-\theta}_R$; this is used in the following proof.

\begin{proof}
Let $u\in N^{1,p}(2\lambda B_R)$ be a function such that $u\ge 1$ on $E$  and $u\le 0$ on $F$.
By the triangle inequality, 
\[
|u(x)-u_{B(x_0,2R)}|\ge \frac12\textrm{ for every }x\in E\quad \textrm{or}\quad |u(y)-u_{B(x_0,2R)}|\ge \frac12 \textrm{ for every }y\in F.
\]
By symmetry, it is enough to consider the first case.


Let $x\in E$ be a Lebesgue point of $u$. Then
\[
\frac12\le \sum_{j \in \mathbb{N}}|u_{B_j(x)}-u_{B_{j+1}(x)}|,
\]
where $B_1(x)=B(x_0,2R)$ and for $j\ge 2$, $B_j(x)=B(x,2^{2-j}R)$.

Noting that $B_{j+1}(x)\subset B_j(x)\subset 2B_R$ for each positive integer $j$, from the doubling property of $\mu_Y$ followed by the Poincar\'e inequality, it follows that
	\begin{align*}
		\frac12\le \sum_{j\in\N}|u_{B_j(x)}-u_{B_{j+1}(x)}|
		\lesssim & \sum_{j\in\N}\vint_{B_j(x)}|u-u_{B_j(x)}|\, d\mu_Y\\
		\lesssim & \sum_{j\in\N} 2^{-j}R\, \left(\vint_{\lambda B_j(x)}g_u^p\, d\mu_Y\right)^{1/p}\\
		\lesssim & \sum_{j\in\N}\frac{2^{-j}R}{\mu_Y(B_j(x))^{1/p}}\, \left(\int_{\lambda B_j(x)}g_u^p\, d\mu_Y\right)^{1/p}.
	\end{align*} 
Thus we have for any $\eta>0$,
	\begin{align*}
		\frac12 c(\eta)\, \sum_{j\in\N}2^{-j\eta}=
		\frac12\lesssim & \sum_{j\in\N}\frac{2^{-j}R}{\mu_Y(B_j(x))^{1/p}}\, \left(\int_{\lambda B_j(x)}g_u^p\, d\mu_Y\right)^{1/p}.
	\end{align*}
All the comparison constants implicitly referred to above depend solely on the doubling constant and the constants associated with the Poincar\'e inequality.
It follows that there is a positive integer $j_x$ such that
	\[
	\frac{c(\eta)}{2}\, 2^{-j_x\eta}
	\lesssim \frac{2^{-j_x}R}{\mu_Y(B_{j_x}(x))^{1/p}}\, \left(\int_{\lambda B_{j_x}(x)}g_u^p\, d\mu_Y\right)^{1/p},
	\]
	that is,
	\[
	2^{-j_x(\eta p-p+\theta)}\, \frac{\mu_Y(B_{j_x}(x))}{(2^{-j_x}R)^\theta}
	   \lesssim c(\eta)^{-p}\, R^{p-\theta}\, \int_{\lambda B_{j_x}(x)}g_u^p\, d\mu_Y.
	\]
	Choosing $\eta=1-\tfrac{\theta}{p}>0$ in the above analysis, we get
	\[
	\frac{\mu_Y(B_{j_x}(x))}{\rad(B_{j_x}(x))^\theta}\lesssim R^{p-\theta}\, \int_{\lambda B_{j_x}(x)}g_u^p\, d\mu_Y,
	\]
	where the comparison constant now also depends on $c(\eta)$ corresponding to the choice of $\eta$ made above, and so on $p$ and $\theta$.
	
	Let $E_0 := \{x \in E : x \text{ is a Lebesgue point of } u \}$.
	Then the collection $\lambda B_{j_x}(x)$, $x\in E_0$, is a cover of $E_0$. Thanks to the 5-covering lemma~\cite{Hei}, we obtain a countable pairwise disjoint subcollection $\{B_k\}_{k\in I\subset \N}$ such that $\{5B_k\}_{k\in I}$ is a cover of $E_0$. 
Here each $B_k$ is of the form $\lambda B_{j_x}(x)$ for some $x\in E_0$.

Note that $p>\theta$, and so by~\cite[Proposition~3.11]{GibaraKorteSh} and by \cite[Theorem~4.1]{KKST} for the case $p=1$ and \cite[Theorem 9.2.8]{HKST} for the case $p>1$, we know that $\mathcal{H}^{-\theta}$-a.e.~point is a Lebesgue point of $u$.
Hence the set of points $x\in E$ that are not Lebesgue points of $u$ forms 
a $\mathcal{H}^{-\theta}_R$-null set, and so it follows from Lemma~\ref{lem:Karadharai} and the doubling property of $\mu_Y$ that
\begin{align*}
M\, \frac{\mu_Y(B(x_0,R))}{R^\theta} \le \mathcal{H}^{-\theta}_R(E)\lesssim
\mathcal{H}^{-\theta}_{10\lambda R}(E)
&\le \sum_{k\in I}\frac{\mu_Y(5B_k)}{\rad(5 B_k)^\theta}\\
&\lesssim \sum_{k\in I}\frac{\mu_Y(\tfrac{1}{\lambda}\,B_k)}{\rad(\frac{1}{\lambda} B_k)^\theta}\\
&\lesssim R^{p-\theta}\, \sum_{k\in I}\int_{B_k}g_u^p\, d\mu_Y
\le R^{p-\theta}\, \int_{2 \lambda B_R} g_u^p\, d\mu_Y,
\end{align*}
that is,
\begin{equation*}
M\, \frac{\mu_Y(B(x_0,R))}{R^p}\lesssim \int_{2\lambda B_R}g_u^p\, d\mu_Y.
\end{equation*}
By taking the infimum over all such $u$, we conclude the proof of the proposition.
\end{proof}

The next lemma follows from~\cite[Proposition~3.11]{GibaraKorteSh} by using a similar proof as the proof of~\cite[Lemma~8.1]{GibaraKorteSh}.
The lemma provides a useful tool connecting $p$-capacity null subsets of $\partial\Om$ (and $\partial_\rho\Om$) with $\nu$-null subsets ($\nu_\rho$-null subsets respectively).

\begin{lemma}\label{lem:capacity2nu}
Let $p \geq 1$ and  $(Y,d_Y,\mu_Y)$ be a metric measure space with $\mu_Y$ a doubling measure supporting a $p$-Poincar\'e inequality.
Suppose also that $E$ is a Borel subset of $Y$ and that $\nu_E$ is a Borel measure supported on $E$, and that there exists $0 < \theta < p$ so that for each $x\in E$ and $0<r<2\diam(E)$ we have $\nu_E(B(x,r))\simeq r^{-\theta}\, \mu_Y(B(x,r))$. Then if $A\subset E$ such that $\cp_p(A)=0$, we must have $\nu_E(A)=0$.
\end{lemma}

We now turn to the definition of one of the nonlocal function spaces of interest in this paper.

\begin{definition}\label{def:Besov}
Let $(Y,d_Y,\mu_Y)$ be a metric measure space.
For $p \geq 1$, $0 < \Theta < 1$ and $u$ a measurable function in $Y$, we consider the Besov energy
\begin{equation*}
\|u\|_{B^{\Theta}_{p,p}(Y, d_Y,\mu_Y)}^p
:=\int_Y \int_Y \frac{|u(y)-u(x)|^p}{\mu_Y(B(x,d_Y(x,y)))\, d_Y(x,y)^{p\Theta}}\, d\mu_Y(y)\, d\mu_Y(x).
\end{equation*}
We say that $u$ is in the \emph{homogeneous Besov space} $\HB^{\Theta}_{p,p}(Y,d_Y,\mu_Y)$, if $\|u\|_{B^{\Theta}_{p,p}(Y,d_Y,\mu_Y)}$ is finite.
We say that $u$ is in the \emph{Besov space} $B^{\Theta}_{p,p}(Y,d_Y,\mu_Y)$, if $\|u\|_{B^{\Theta}_{p,p}(Y,d_Y,\mu_Y)}$ is finite and $\int_Y |u|^p\, d\mu_Y$ is also finite.
We define the norm on $B^{\Theta}_{p,p}(Y,d_Y,\mu_Y)$ as $\| u \|_{B^{\Theta}_{p,p}(Y,d_Y,\mu_Y)} + \| u \|_{L^p(Y,\mu_Y)}$ whenever $u\in B^\Theta_{p,p}(Y,d_Y,\mu_Y)$.
\end{definition}

While the above notion does make sense even when $\Theta\ge 1$, unless $B^\Theta_{p,p}(Y)$ consists only of constant functions, it does not arise as the trace of a Newton-Sobolev class of functions on a metric measure space when $\Theta\ge 1$.
It is possible for $B^\Theta_{p,p}(Y)$ to have non-constant functions even for some $\Theta\ge 1$, as in the case of the Sierpinski carpet, see for instance~\cite{KajinoShimizu, KumagaiShShimizu}. In the present paper we are interested in Besov spaces that arise as traces of Newton-Sobolev spaces, and so here $0<\Theta<1$.

Recall that in this paper we have the following standing assumptions: the complete metric space $(X,d)$ has a doubling measure $\mu$ and supports a $p$-Poincar\'e inequality. We assume that $\Om \subset X$ is an unbounded uniform domain and $\overline\Om^d = X$.
We assume that the boundary $\partial \Om$ is unbounded, uniformly perfect and equipped with the measure $\nu$ that is $\theta$-codimensional with respect to $\mu$. Finally we assume that the metric density function $\rho$ is lower semicontinuous and satisfies conditions \ref{condA}, \ref{condB}, \ref{condC}, and \ref{condD}. Thus we have the following structural data:
$C_U$, $C_{\mu}$, $p$, $\lambda$, $C_P$, $\theta$, $C_{\theta}$, $\kappa$, $C_A$, $C_B$, $C_C$ and $C_D$.

\section{Conformal deformations of the metric and measure}
\label{sect-spher-def}

The focus of this section is to describe the construction of deformations of unbounded domains into bounded domains. Much of the construction here was first considered in~\cite{KRST}. 
There a parameter $\sigma$ was introduced in the deformation of the measure, but in the present paper we make the choice of $\sigma=p$. 

As pointed out at the end of the previous section, $(X, d,\mu)$ is a complete metric measure space and $\Omega$ is a uniform domain in this metric space with $\partial\Om$ unbounded. Thus $\Omega$ is rectifiably path-connected, that is, for each pair $x,y\in\Om$ there is a rectifiable curve $\gamma_{xy}$ in $\Om$ with end points $x$ and $y$.
While it is part of the structural assumptions in this paper that $\Om$ is a uniform domain and that $\mu$ is a doubling measure on $(\Om,d)$, with $(\Om,d,\mu)$ supporting a $p$-Poincar\'e inequality for some fixed $1\le p<\infty$, the definitions of $d_\rho$ and $\mu_\rho$ make sense even if $\Om$ is only rectifiably path-connected.
Recall that we assume that $X = \overline\Omega^d$.

We fix a base point $b\in\partial\Omega$ and for each $x\in \overline\Omega^d$ we set $|x| := d(x,b)$. We consider a lower semicontinuous metric density function $\rho : (0,\infty) \to (0,\infty)$.
The new metric $d_\rho$ on $\Om$ is defined for any $x , y \in \Omega$ by 
\begin{equation*}
d_{\rho}(x,y) := \inf_{\gamma \in \Gamma(x,y)} \int_{\gamma} \rho(|\cdot|) ds,
\end{equation*}
where $\Gamma(x,y)$ is the set of rectifiable curves in $\Omega$ with end points $x$ and $y$.
Note that $\overline\Omega^{d_{\rho}}$ is the collection of equivalence classes of Cauchy sequences in $(\Om,d_{\rho})$, with sequences converging to a point in $\Om$ identified with that point. Thus $\Om\subset \overline\Omega^{d_{\rho}}$, and so we define $\partial_{\rho} \Omega := \overline\Omega^{d_{\rho}} \setminus \Omega$. Given two $d_\rho$--Cauchy sequences $(x_i)_{i=1}^\infty$ and $(y_i)_{i=1}^\infty$ representing points $x,y\in \overline\Omega^{d_{\rho}}$, we set $d_\rho(x,y)=\lim_{i\to\infty}d_{\rho}(x_i,y_i)$. This number is then independent of the sequences representing the two points and this $d_\rho$ defines a metric on $\overline\Omega^{d_{\rho}}$. If both $x,y\in\Om$, then this notion agrees with $d_\rho$ on $\Om$.

The new measure on $\Omega$ is defined by 
\[
  d\mu_{\rho}(x) := \rho(|x|)^p d\mu(x).
\]
By assuming conditions \ref{condA}, \ref{condB} and \ref{condC} from the structural assumptions on $\rho$ listed in Section~\ref{section:intro}, we obtain from~\cite{KRST} that $(\Omega,d_{\rho})$ is a bounded metric space that is uniform and $\mu_{\rho}$ is a doubling measure in this space.
Additionally, we get that the space $(\Omega,d_{\rho},\mu_{\rho})$ supports a $p$-Poincar\'e inequality.
Also the completion of $\Omega$ with respect to the metric $d_{\rho}$ adds exactly one point to the completion of $\Omega$ with respect to $d$. This point is called the point at infinity and it is denoted $\infty$.
Finally we get that $g$ is a $p$-integrable upper gradient of $u$ in $(\Omega,d_{\rho},\mu_{\rho})$ if and only if $\rho(|\cdot|) g(\cdot)$ is a $p$-integrable upper gradient of $u$ in $(\Omega,d,\mu)$.

Recall from our standing assumptions that $\partial\Om$ is equipped with a Radon measure $\nu$ which is $\theta$-codimensional with respect to the measure $\mu$.
We now turn to transforming the measure $\nu$ on the boundary $\partial\Om$ to obtain a measure on $\partial_\rho\Om=\partial\Om\cup\{\infty\}$; such a transformation was not considered in~\cite{KRST}.
We define the new sphericalized measure on the boundary $\partial_{\rho} \Omega$ by setting 
\[
d\nu_{\rho}(x) := \rho(|x|)^{p-\theta} d\nu(x)
\] 
and $\nu_{\rho}(\{b\}) := \nu_{\rho}(\{\infty\}) := 0$. Notice that $\nu(\{b\})=0$, because $\nu$ is doubling in $(\partial \Omega ,d)$ and $(\partial \Omega ,d)$ is uniformly perfect, see~\cite[Lemma 2.2]{bjornbjornkorterogovintakala} and \cite[Theorem~1]{MaciasSegovia}.

The doubling property of $\nu$ and Condition \ref{condA} give us that
\begin{equation}
\label{condDreverse}
\int_{\partial \Omega \setminus B(b,r)} \rho(|x|)^{p-\theta} d\nu(x)
\simge \rho(r)^{p-\theta} \nu(B(b,r+1)),
\end{equation}
where the comparison constant depends on $C_{\mu}$, $p$, $\theta$, $C_{\theta}$, $\kappa$ and $C_A$.
Similarly the doubling property of $\mu$ and Condition \ref{condA} give us that
\begin{equation}
\label{condCreverse}
\int_{\Omega \setminus B(b,r)} \rho(|x|)^p d\mu(x)
\simge \rho(r)^p \mu(B(b,r+1)),
\end{equation}
where the comparison constant depends on $C_{\mu}$, $p$, and $C_A$.

\begin{lemma}
Condition~\ref{condD} is equivalent to the condition that there is a constant $C_D^\prime$ with
\begin{equation*} 
\int_{\Omega \setminus B(b,r)} \rho(|x|)^{p-\theta} (|x|+1)^{-\theta} d\mu(x)
\le 
C_D^\prime \, \rho(r)^{p-\theta} (r+1)^{-\theta} \mu(B(b,r+1))
\end{equation*}
for every $r>0$.
\end{lemma}

\begin{proof}
The comparability of the right-hand sides of the above condition and of Condition~\ref{condD} follows immediately from the codimension property.
To prove the comparability of the left-hand sides of the above condition and of Condition~\ref{condD}, assume first that $r \geq 1$.
Then we have 
\[
\nu \left( B \left( b,(2 \kappa)^j r \right) \setminus B \left( b,(2 \kappa)^{j-1} r \right) \right) \simge \nu \left( B \left( b,(2 \kappa)^j r \right) \right)
\] 
for every positive integer $j$, where $\kappa$ is the uniform perfectness constant associated with $(\partial\Om,d)$.
This follows from the uniform perfectness of $(\partial \Omega ,d)$ and the doubling property of $\nu$.
Then by condition \ref{condA},
\begin{align*}
\int_{\partial \Omega \setminus B(b,r)} \rho(|x|)^{p-\theta} d\nu(x)
&= \sum_{j=1}^{\infty} \int_{B \left( b,(2 \kappa)^j r \right) \setminus B \left( b,(2 \kappa)^{j-1} r \right)} \rho(|x|)^{p-\theta} d\nu(x)
\\
&\simeq \sum_{j=1}^{\infty} \nu \left( B \left( b,(2 \kappa)^j r \right) \setminus B \left( b,(2 \kappa)^{j-1} r \right) \right) \rho \left( (2 \kappa)^j r \right)^{p-\theta}
\\
&\simeq \sum_{j=1}^{\infty} \nu \left( B \left( b,(2 \kappa)^j r \right) \right) \rho \left( (2 \kappa)^j r \right)^{p-\theta}.
\end{align*}
Similarly we have
\begin{align*}
\int_{\Omega \setminus B(b,r)} \rho(|x|)^{p-\theta} (|x|+1)^{-\theta} d\mu(x)
&\simeq \sum_{j=1}^{\infty} \int_{B \left( b,(2 \kappa)^j r \right) \setminus B \left( b,(2 \kappa)^{j-1} r \right)} \rho(|x|)^{p-\theta} |x|^{-\theta} d\mu(x)
\\
\simeq &\sum_{j=1}^{\infty} \mu \left( B \left( b,(2 \kappa)^j r \right) \setminus B \left( b,(2 \kappa)^{j-1} r \right) \right) \rho \left( (2 \kappa)^j r \right)^{p-\theta} \left( (2 \kappa)^j r \right)^{-\theta}
\\
\simeq &\sum_{j=1}^{\infty} \mu \left( B \left( b,(2 \kappa)^j r \right) \right) \rho \left( (2 \kappa)^j r \right)^{p-\theta} \left( (2 \kappa)^j r \right)^{-\theta}.
\end{align*}
Note that $\Omega$ is connected and $\mu$ is doubling in $\Omega$.
Finally the comparability follows from the codimensionality.
On the other hand if $0 < r < 1$, then
\begin{align*} 
\int_{B(b,1) \setminus B(b,r)} \rho(|x|)^{p-\theta} d\nu(x)
&\simle \rho(1)^{p-\theta} \nu(B(b,1))
\simeq \rho(1)^{p-\theta} \nu(B(b,2\kappa) \setminus B(b,1))
\\
&\simeq \int_{B(b,2\kappa) \setminus B(b,1)} \rho(|x|)^{p-\theta} d \nu(x)
\leq \int_{\partial \Omega \setminus B(b,1)} \rho(|x|)^{p-\theta} d \nu(x),
\end{align*}
where we used condition \ref{condA}, the doubling property of $\nu$ and the uniform perfectness of $(\partial\Omega,d)$.
Similarly
\begin{align*} 
\int_{B(b,1) \setminus B(b,r)} \rho(|x|)^{p-\theta} (|x|+1)^{-\theta} d\mu(x)
&\simle \rho(1)^{p-\theta} \mu(B(b,1))
\simeq \rho(1)^{p-\theta} \mu(B(b,2) \setminus B(b,1))
\\
&\simeq \int_{B(b,2) \setminus B(b,1)} \rho(|x|)^{p-\theta} (1+|x|)^{-\theta} d \mu(x)
\\
&\leq \int_{\Omega \setminus B(b,1)} \rho(|x|)^{p-\theta} (1+|x|)^{-\theta} d \mu(x).
\end{align*}
Now, applying the comparability obtained from the case $r\ge 1$ considered earlier with the choice of $r=1$, together with the two previous inequalities, we get
\begin{align*} 
\int_{\partial \Omega \setminus B(b,r)} \rho(|x|)^{p-\theta} d\nu(x)
&= \int_{\partial \Omega \setminus B(b,1)} \rho(|x|)^{p-\theta} d\nu(x) + \int_{B(b,1) \setminus B(b,r)} \rho(|x|)^{p-\theta} d\nu(x)
\\
&\simeq \int_{\partial \Omega \setminus B(b,1)} \rho(|x|)^{p-\theta} d\nu(x)
\simeq \int_{\Omega \setminus B(b,1)} \rho(|x|)^{p-\theta} (|x|+1)^{-\theta} d\mu(x)
\\
&\simeq \int_{\Omega \setminus B(b,r)} \rho(|x|)^{p-\theta} (|x|+1)^{-\theta} d\mu(x).\qedhere
\end{align*}
\end{proof}

\section{The codimensionality of \texorpdfstring{$\mu_{\rho}$}{} and \texorpdfstring{$\nu_{\rho}$}{}}
\label{sect-codim}

We start this section by proving that $\nu_{\rho}$ is a doubling measure on $(\partial_{\rho}\Omega,d_{\rho})$. 
Then by using the doubling properties of $\mu_{\rho}$ and $\nu_{\rho}$ we prove that these measures satisfy the codimensional relationship.
To prove that $\nu_{\rho}$ is doubling in $(\partial_{\rho} \Omega,d_{\rho})$ we use the results of \cite{KRST}.
Balls in the metric $d_{\rho}$ are denoted by $B_{\rho}$, meaning that
\begin{equation*}
B_{\rho}(x,r)
:= \left\{ y \in \overline\Omega^{d_{\rho}} : d_{\rho}(x,y) < r \right\}
\end{equation*}
for every $x \in \overline\Omega^{d_{\rho}}$ and $r > 0$.

The following two lemmas have analogs in~\cite{KRST}, where the points considered were assumed to be in $\Om$.
However, thanks to Condition~\ref{condA}, we can conclude the corresponding results for points in the boundary $\partial \Om$ as well via continuity of $d$ and $d_\rho$. 
We omit their proofs but include reference to the corresponding results in~\cite{KRST}.
Note that we do not let $x$ equal $b$ simply because $\rho(0)$ is not defined. 
However, $d_\rho(b,y)$ is well-defined as the extension of $d_\rho$ from $\Om$ to its metric completion $\Om\cup\partial\Om\cup\{\infty\}$ with $b\in\partial\Om$.

\begin{lemma}\label{lemma4.1}
Whenever $x,y \in \overline{\Omega}^d$ satisfy $x \neq b$ and $\frac{1}{2}(|x|+1) \leq |y|+1 \leq 2(|x| + 1)$, then
\begin{equation*}
d_{\rho}(x,y)
\simeq \rho(|x|) d(x,y),
\end{equation*}
where the comparison constants depend only on $C_U$, $C_A$ and $C_B$.
\end{lemma}
We can infer the above lemma from~\cite[Lemma~3.11]{KRST} by using the doubling property of $\rho$ and approximating points in $\partial\Om$ by points in $\Om$.

\begin{lemma}\label{lemma4.2}
Let $x,y \in \overline{\Omega}^d$ such that $x \neq b$ and $|y| \geq 2 |x| + 1$. Then
\begin{equation*}
d_{\rho}(x,y)
\simeq \rho(|x|) (|x|+1),
\end{equation*}
where the comparison constants depend only on $C_U$, $C_A$ and $C_B$.
In particular there exist constants $C_1$ and $C_2$ depending on $C_U$, $C_A$ and $C_B$ such that
\begin{equation*}
C_1 \rho(|x|) (|x|+1)
\leq d_{\rho}(x,\infty)
\leq C_2 \rho(|x|) (|x|+1).
\end{equation*}
\end{lemma}
The proof is very similar to the proof of~\cite[Lemma 3.12]{KRST}, and is therefore omitted here.

Let us define $h\colon (0,\infty)\to (0,\infty)$ by setting $h(t):=(t+1)\rho(t)$ the same way as in \cite{KRST}. The inverse of $h$ is defined by
\[
h^{-1}(\tau):=\inf \{t > 0: h(t)\le \tau\},\quad \tau>0.
\]
Notice that $h^{-1}$ is monotone decreasing and unbounded (since $h(t)\to0$ as $t\to\infty$, see~\cite[Remark~2.1]{KRST}).
Note that when $0 < \tau \leq \rho(1)/C_A$, we know that $h^{-1}(\tau)\ge 1>0$.
From \cite[Lemma 4.2]{KRST} we get that $h^{-1}(\tau) \simeq h^{-1}(\tau/2)$ for every $0 < \tau \leq \rho(1) / C_A$ with comparison constants depending only on $C_A$ and $C_B$.
Then, with $C_1$ and $C_2$ as in Lemma~\ref{lemma4.2}, we get that for every $0 < r \leq C_1  \rho(1) / C_A$,
\begin{equation}
\label{balls-at-infinity}
\overline\Omega^{d_{\rho}} \setminus B \left( b,2h^{-1} \left( \frac{r}{2 C_2 C_A C_B} \right) \right)
\subset B_\rho(\infty,r)
\subset \overline\Omega^{d_{\rho}} \setminus B \left( b,h^{-1} \left( \frac{r}{C_1} \right) \right) .
\end{equation}
The proof of \eqref{balls-at-infinity} is similar to the proof of \cite[Lemma 4.4]{KRST} so it is omitted.

\begin{lemma}
\label{small-balls}
There exist positive constants $c_0$, $a_1$ and $a_2$ depending only on $C_U$, $C_A$ and $C_B$ such that if $x \in \overline\Omega^{d} \setminus \{b\}$ and $0 < r \leq 2c_0 d_{\rho}(x,\infty)$, then
\begin{equation*}
B \left( x,\frac{a_1 r}{\rho(|x|)} \right) \setminus \{b\}
\subset B_{\rho}(x,r) \setminus \{b\}
\subset B \left( x,\frac{a_2 r}{\rho(|x|)} \right).
\end{equation*}
Also $\frac{1}{C_A}\rho(|x|) \leq \rho(|y|) \leq C_A \rho(|x|)$ for every $y \in B_{\rho}(x,r) \setminus \{b\}$.
\end{lemma}

The proof is the same as the proof of~\cite[Lemma 4.9]{KRST}, and so it is omitted.

\begin{lemma}
\label{doubling-at-infinity}
For all $r>0$ we have
$$
0<\nu_{\rho}(B_\rho(\infty,2r))
\simle \nu_{\rho}(B_\rho(\infty,r))<\infty,
$$
with the comparison constant depending on $C_U$, $C_{\mu}$, $p$, $\theta$, $C_{\theta}$, $\kappa$, $C_A$, $C_B$ and $C_D$.
\end{lemma}
\begin{proof} 
From Conditions \ref{condA} and \ref{condD} we get that
\begin{equation*}
\nu_{\rho}(\partial\Omega \setminus B(b,r))
\leq C_D \rho(r)^{p-\theta} \nu(B(b,r+1))
\leq C_D C_A^{p-\theta} \rho(1)^{p-\theta} \nu(B(b,2))
\end{equation*}
for every $0 < r \leq 1$.
Thus by taking the limit $r \to 0$, we get that
\begin{equation}
\label{measure-of-boundary}
\nu_{\rho}(\partial_{\rho}\Omega)
\leq C_D C_A^{p-\theta} \rho(1)^{p-\theta} \nu(B(b,2))
< \infty
\end{equation}
and this shows the final inequality.
Moreover, since $\rho>0$ and $\nu(B_\rho(\infty,r))>0$ by \eqref{balls-at-infinity}, we have that $\nu_{\rho}(B_\rho(\infty,r))>0$.

If $0<r\le \frac{C_1  \rho(1)}{2C_A}$, then by using \eqref{balls-at-infinity}, 
condition \ref{condD}, condition \ref{condA}, the doubling properties of $h^{-1}$ and $\nu$, and \eqref{condDreverse} we get
\begin{align*}
\nu_{\rho}(B_{\rho}(\infty,2r))
&\leq \nu_{\rho} \left( \overline\Omega^{d_{\rho}} \setminus B(b,h^{-1}(2 r / C_1)) \right)
= \int_{\partial\Omega \setminus B(b,h^{-1}(2r/C_1))} \rho(|x|)^{p-\theta} d\nu(x)
\\
&\leq C_D \rho(h^{-1}(2r/C_1))^{p-\theta} \nu(B(b,h^{-1}(2r/C_1)+1))
\\
&\simeq \rho \left( 2h^{-1} \left( \frac{r}{2 C_2 C_A C_B} \right) \right)^{p-\theta} \nu \left( B \left( b,2h^{-1} \left( \frac{r}{2 C_2 C_A C_B} \right) +1 \right) \right)
\\
&\simle \int_{\partial\Omega \setminus B \left( b,2h^{-1} \left( \frac{r}{2 C_2 C_A C_B} \right) \right) } \rho(|x|)^{p-\theta} d\nu(x)
\\
&= \nu_{\rho} \left( \partial_{\rho}\Omega \setminus B \left( b,2h^{-1} \left( \frac{r}{2 C_2 C_A C_B} \right) \right) \right)
\leq \nu_{\rho}(B_{\rho}(\infty,r)).
\end{align*}
If $r\ge \frac{C_1  \rho(1)}{2 C_A}$, then similarly we get
\begin{equation*}
\nu_{\rho}(B_\rho(\infty,r))
\geq \nu_{\rho} \left( B_\rho \left( \infty,\frac{C_1  \rho(1)}{2 C_A} \right) \right)
\simge \nu_{\rho}(\partial_{\rho}\Omega)
\geq \nu_{\rho}(B_\rho(\infty,2r)).
\end{equation*}
Here we also used \eqref{measure-of-boundary} and the fact that $h^{-1}\left( \frac{\rho(1)}{C_A} \right) \simeq 1$, which follows from \cite[Lemma 2.2]{KRST}.
\end{proof}

\begin{theorem}\label{thm:nu-rho-doubling}
The measure $\nu_{\rho}$ is doubling in the space $(\partial_{\rho}\Omega,d_{\rho})$ i.e. there exists a constant $C_{\nu_{\rho}} \geq 1$ that depends on $C_U$, $C_{\mu}$, $p$, $\theta$, $C_{\theta}$, $\kappa$, $C_A$, $C_B$ and $C_D$ such that
\begin{equation*}
0
< \nu_{\rho} (B_{\rho}(x,2r))
\leq C_{\nu_{\rho}} \nu_{\rho} (B_{\rho}(x,r))
< \infty
\end{equation*}
for every $x \in \partial_{\rho}\Omega$ and $r > 0$.
\end{theorem}

The proof of this theorem is similar to that of~\cite[Theorem 4.10]{KRST}, but for the reader's convenience we provide a proof here, especially since the curves used to construct the metric $d_\rho$ on $\partial_\rho\Om$ do not necessarily lie in $\partial_\rho\Om$.

\begin{proof}
This is already shown for the case $x = \infty$. If we show this for points other than $b$, then the result also follows for $x = b$ by approximating $b$ with a sequence in $\partial\Omega$, which is possible, because $(\partial\Omega,d)$ is uniformly perfect.

Thus let us assume $x \in \partial\Omega \setminus \{b\}$.
We get the final inequality from \eqref{measure-of-boundary}.
For the first inequality we get from Lemma \ref{small-balls} that
\begin{equation*}
\nu_{\rho}(B_{\rho}(x,r))
\geq \nu_{\rho}(B_{\rho}(x,\min\{r,2c_0d_{\rho}(x,\infty)\}))
\geq \nu_{\rho} \left( B \left( x,\frac{a_1 \min\{r,2c_0d_{\rho}(x,\infty)\}}{\rho(|x|)} \right) \right).
\end{equation*}
This is positive, because $\rho$ is positive and the $\nu$-measure of the ball is positive, because
\begin{equation*}
a_1 \min\{r,2c_0d_{\rho}(x,\infty)\} / \rho(|x|)
> 0.
\end{equation*}

The rest of the proof is divided into the following three cases.
\begin{enumerate}
\item[(1)] $0 < r \leq c_0 d_{\rho}(x,\infty)$
\item[(2)] $r \geq 2 d_{\rho}(x,\infty)$
\item[(3)] $c_0 d_{\rho}(x,\infty) \leq r \leq 2 d_{\rho}(x,\infty)$
\end{enumerate}
Here $c_0$ is the constant in Lemma \ref{small-balls}.

In Case~(1) we get from Lemma \ref{small-balls} and the doubling property of $\nu$ that
\begin{align*}
\nu_{\rho}(B_{\rho}(x,2r))
&\simeq \rho(|x|)^{p-\theta} \nu(B_{\rho}(x,2r))
\leq \rho(|x|)^{p-\theta} \nu \left( B \left( x, \frac{2 a_2 r}{\rho(|x|)} \right) \right)
\\
&\simeq \rho(|x|)^{p-\theta} \nu \left( B \left( x, \frac{a_1 r}{\rho(|x|)} \right) \right)
\leq \rho(|x|)^{p-\theta} \nu(B_{\rho}(x,r))
\simeq \nu_{\rho}(B_{\rho}(x,r)).
\end{align*}
This completes the proof of Case~(1).

For Case~(2) we have by Lemma \ref{doubling-at-infinity}
\begin{equation*}
\nu_{\rho}(B_{\rho}(x,2r))
\leq \nu_{\rho}(B_{\rho}(\infty,3r))
\simle \nu_{\rho}(B_{\rho}(\infty,r/2))
\leq \nu_{\rho}(B_{\rho}(x,r)).
\end{equation*}

Now we move on to Case~(3). We shall show that
\begin{equation*}
\nu_{\rho}(B_{\rho}(x,r))
\simeq \rho(|x|)^{p-\theta} \nu(B(x,|x|+1)),
\end{equation*}
if $c_0 d_{\rho}(x,\infty) \leq r \leq 4 d_{\rho}(x,\infty)$. Since the right hand side does not depend on $r$, this will complete the proof.

We start by proving the "$\simle$" part. Thus let us assume that $y \in B_{\rho}(x,r) \setminus \{b,\infty\}$.
The first step is to show that there exists a constant $M \geq 2$ that depends only on $C_U$, $C_A$ and $C_B$, such that $|x|+1 \leq M(|y|+1)$. If $|x|+1 \leq 2(|y|+1)$, the claim holds for any $M \geq 2$.
If $|x|+1 \geq 2(|y|+1)$ i.e. $|x| \geq 2|y|+1 \geq |y|$, then we get  by Lemma \ref{lemma4.2} that
\begin{equation*}
\rho(|y|)(|y|+1)
\simeq d_{\rho}(x,y)
< 4 d_{\rho}(x,\infty)
\simeq \rho(|x|)(|x|+1)
\simle \left( \frac{|y|+1}{|x|+1} \right)^{\frac{1}{C_A C_B}} (|y|+1) \rho(|y|),
\end{equation*}
where we also used the fact that $d_{\rho}(x,y) < r \leq 4d_{\rho}(x,\infty)$ and \cite[Lemma 2.2]{KRST}. Therefore there exists a constant $M$ such that $|x|+1 \leq M(|y|+1)$.

The next step is to deal separately with the cases $|x|+1 \leq M$ and $|x|+1 > M$. We start with $|x|+1 \leq M$. In this case we get from Condition \ref{condA} that $\rho(1) \simeq \rho(|x|)$ and therefore by \eqref{measure-of-boundary} and the doubling property of $\nu$,
\begin{align*}
\nu_{\rho}(B_{\rho}(x,r))
\leq \nu_{\rho}(\partial_{\rho}\Omega)
\simle \rho(1)^{p-\theta} \nu(B(b,2))
\simle \rho(|x|)^{p-\theta} \nu(B(b,1)) 
\leq\rho(|x|)^{p-\theta} \nu(B(x,|x|+1)).
\end{align*}

Now let $|x|+1 > M$.
Let us denote $r_x := \frac{1}{M}(|x|+1) - 1 > 0$. Then $|y| \geq r_x$ for every $y \in B_{\rho}(x,r) \setminus \{b,\infty\}$ and $\rho(|x|) \simeq \rho(r_x)$ by Condition \ref{condA}. Thus by Condition \ref{condD},
\begin{align*}
\nu_{\rho}(B_{\rho}(x,r))
&= \int_{B_{\rho}(x,r)} \rho(|y|)^{p-\theta} d \nu(y)
\leq \int_{\partial\Omega \setminus B(b,r_x)} \rho(|y|)^{p-\theta} d \nu(y)
\leq C_D \rho(r_x)^{p-\theta} \nu(B(b,r_x+1))
\\
&= C_D \rho(r_x)^{p-\theta} \nu \left( B \left( b,\frac{1}{M}(|x|+1) \right) \right)
\simeq \rho(|x|)^{p-\theta} \nu(B(x,|x|+1)).
\end{align*}
Thus we have shown the "$\simle$" part.

Now let us prove the "$\simge$" part.
We show that there exists a positive constant $a$ depending only on $C_U$, $C_A$ and $C_B$ such that 
\begin{equation*}
B(x,a(|x|+1)) \setminus \{b\}
\subset B_{\rho}(x,r).
\end{equation*}
So assume that $y \in B(x,a(|x|+1)) \setminus \{b\}$.
Then by the triangle inequality,
\begin{equation*}
\frac{1}{2}(|x|+1)
\leq |y|+1
\leq 2(|x|+1),
\end{equation*}
if we set $a \leq 1/2$.
Thus by Condition \ref{condA} we have $\rho(|y|) \simeq \rho(|x|)$.
Then by Lemmas~\ref{lemma4.1} and~\ref{lemma4.2}, 
\begin{equation*}
d_{\rho}(x,y)
\simeq \rho(|x|) d(x,y)
< a \rho(|x|) (|x|+1)
\simeq a c_0 d_{\rho}(x,\infty)
\leq a r.
\end{equation*}
Thus by setting $a$ small enough, we get $B(x,a(|x|+1)) \setminus \{b\} \subset B_{\rho}(x,r)$. Therefore,
\begin{align*}
\nu_{\rho}(B_{\rho}(x,r))
&= \int_{B_{\rho}(x,r)} \rho(|y|)^{p-\theta} d \nu(y)
\geq \int_{B(x,a(|x|+1))} \rho(|y|)^{p-\theta} d \nu(y)
\simeq \rho(|x|)^{p-\theta} \nu(B(x,a(|x|+1)))
\\
&\simeq \rho(|x|)^{p-\theta} \nu(B(x,|x|+1)).
\end{align*}
This completes the proof.
\end{proof}

Now we are ready to prove that $\nu_{\rho}$ and $\mu_{\rho}$ have the codimension property.

\begin{theorem}\label{thm:nu-rho-codim}
There exists a constant $C_{\theta,\rho} \geq 1$ that depends on $C_U$, $C_{\mu}$, $p$, $\theta$, $C_{\theta}$, $\kappa$, $C_A$, $C_B$, $C_C$ and $C_D$ such that
\begin{equation*}
\frac{1}{C_{\theta,\rho}} \frac{\mu_{\rho}(B_{\rho}(x,r))}{r^\theta}
\leq \nu_{\rho}(B_{\rho}(x,r))
\leq C_{\theta,\rho} \, \frac{\mu_{\rho}(B_{\rho}(x,r))}{r^\theta}
\end{equation*}
for every $x \in \partial_{\rho}\Omega$ and $0<r<2\diam_{\rho}(\partial_{\rho} \Omega)$.
\end{theorem}

\begin{proof}
If we show this result for any $x \in \partial_{\rho}\Omega \setminus \{b\}$, then it also follows for $x=b$ by approximating $b$ with a sequence in $\partial\Omega$, which is possible, because $(\partial\Omega,d)$ is uniformly perfect. 
Thus let us assume that $x \neq b$. The proof is divided into four cases.
\begin{enumerate}
\item[(1)] $0<r \leq 2 c_0 d_{\rho}(x,\infty)$,
\item[(2)] $x = \infty$ and $0<r\leq \frac{C_1 \rho(1)}{C_A}$,
\item[(3)] $x \neq \infty$ and $2 c_0 d_{\rho}(x,\infty) < r \leq \frac{C_1 \rho(1)}{C_A}$,
\item[(4)] $\frac{C_1 \rho(1)}{C_A} < r < 2 \diam_{\rho}(\partial_{\rho} \Omega)$.
\end{enumerate}
Thus we can use \eqref{balls-at-infinity} when proving Case (2).

In Case~(1), we get from Lemma \ref{small-balls} that
\begin{align*}
\mu_{\rho}(B_{\rho}(x,r))
&\simeq \rho(|x|)^p \mu(B_{\rho}(x,r))
\leq \rho(|x|)^p \mu \left( B \left( x,\frac{a_2 r}{\rho(|x|)} \right) \right)
\simeq \rho(|x|)^p \mu \left( B \left( x,\frac{a_1 r}{\rho(|x|)} \right) \right)
\\
&\simeq \rho(|x|)^p \left( \frac{a_1 r}{\rho(|x|)} \right)^{\theta} \nu \left( B \left( x,\frac{a_1 r}{\rho(|x|)} \right) \right)
\simle \rho(|x|)^{p-\theta} r^{\theta} \nu(B_{\rho}(x,r))
\simeq r^{\theta} \nu_{\rho}(B_{\rho}(x,r)),
\end{align*}
where we also used the doubling property of $\mu$ and the codimensionality of $\mu$ and $\nu$.
Similarly
\begin{align*}
\nu_{\rho}(B_{\rho}(x,r))
&\simeq \rho(|x|)^{p-\theta} \nu(B_{\rho}(x,r))
\leq \rho(|x|)^{p-\theta} \nu \left( B \left( x,\frac{a_2 r}{\rho(|x|)} \right) \right)
\\
&\simeq \rho(|x|)^{p-\theta} \left( \frac{a_2 r}{\rho(|x|)} \right)^{-\theta} \mu \left( B \left( x,\frac{a_2 r}{\rho(|x|)} \right) \right)
\simeq \rho(|x|)^p r^{-\theta} \mu \left( B \left( x,\frac{a_1 r}{\rho(|x|)} \right) \right)
\\
&\leq \rho(|x|)^p r^{-\theta} \mu(B_{\rho}(x,r))
\simeq r^{-\theta} \mu_{\rho}(B_{\rho}(x,r)).
\end{align*}

In Case~(2), we use \eqref{balls-at-infinity} and condition \ref{condD} to get
\begin{align*}
\nu_{\rho}(B_{\rho}(x,r))
&\leq \nu_{\rho} \left( \overline\Omega^{d_{\rho}} \setminus B(b,h^{-1}(r/C_1)) \right)
= \int_{\partial\Omega \setminus B(b,h^{-1}(r/C_1))} \rho(|y|)^{p-\theta} d \nu(y)
\\
&\leq C_D \rho(h^{-1}(r/C_1))^{p-\theta} \nu(B(b,h^{-1}(r/C_1)+1))
\\
&\leq C_D C_{\theta} \rho(h^{-1}(r/C_1))^p \rho(h^{-1}(r/C_1))^{-\theta} (h^{-1}(r/C_1)+1)^{-\theta} \mu(B(b,h^{-1}(r/C_1)+1)).
\intertext{Here we also used the codimensionality of $\mu$ and $\nu$. Then by using condition \ref{condA}, the doubling properties of $\mu$ and $h^{-1}$, and the definition of $h$, we obtain}
&\simeq h(h^{-1}(r/C_1))^{-\theta} \rho \left( 2h^{-1} \left( \frac{r}{2 C_2 C_A C_B} \right) \right)^p \mu \left( B \left( b,2h^{-1} \left( \frac{r}{2 C_2 C_A C_B} \right)+1 \right) \right).
\intertext{Notice that $h(h^{-1}(r/C_1)) > \frac{r}{C_1 C_A}$. Thus by using \eqref{condCreverse} we get}
&\simle \left( \frac{r}{C_1 C_A} \right)^{-\theta} \int_{\Omega \setminus B \left( b,2h^{-1} \left( \frac{r}{2 C_2 C_A C_B} \right) \right) } \rho(|y|)^p d \mu(y)
\\
&\simeq r^{-\theta} \mu_{\rho} \left( \Omega \setminus B \left( b,2h^{-1} \left( \frac{r}{2 C_2 C_A C_B} \right) \right) \right)
\leq r^{-\theta} \mu_{\rho}(B_{\rho}(x,r)),
\end{align*}
where in the last step we again used \eqref{balls-at-infinity}.
By using condition \ref{condC} and the doubling property of $\nu$ we get in a similar way
\begin{align*}
\mu_{\rho}(B_{\rho}(x,r))
&\leq \mu_{\rho} \left( \overline\Omega^{d_{\rho}} \setminus B(b,h^{-1}(r/C_1)) \right)
= \int_{\Omega \setminus B(b,h^{-1}(r/C_1))} \rho(|y|)^p d \mu(y)
\\
&\leq C_C \rho(h^{-1}(r/C_1))^p \mu(B(b,h^{-1}(r/C_1)+1))
\\
&\leq C_C C_{\theta} \rho(h^{-1}(r/C_1))^{p-\theta} \rho(h^{-1}(r/C_1))^{\theta} (h^{-1}(r/C_1)+1)^{\theta} \nu(B(b,h^{-1}(r/C_1)+1))
\\
&\simeq h(h^{-1}(r/C_1))^{\theta} \rho \left( 2h^{-1} \left( \frac{r}{2 C_2 C_A C_B} \right) \right)^{p-\theta} \nu \left( B \left( b,2h^{-1} \left( \frac{r}{2 C_2 C_A C_B} \right) +1 \right) \right).
\intertext{Notice that $h(h^{-1}(r/C_1)) \leq \frac{C_A}{C_1} r$. Thus by using \eqref{condDreverse} we get}
&\simle \left( \frac{C_A}{C_1} r \right)^{\theta} \int_{\partial \Omega \setminus B \left( b,2h^{-1} \left( \frac{r}{2 C_2 C_A C_B} \right) \right) } \rho(|y|)^{p-\theta} d \nu(y)
\\
&\simeq r^{\theta} \nu_{\rho} \left( \partial \Omega \setminus B \left( b,2h^{-1} \left( \frac{r}{2 C_2 C_A C_B} \right) \right) \right)
\leq r^{\theta} \nu_{\rho}(B_{\rho}(x,r)).
\end{align*}

In Case~(3), we use the doubling properties of $\mu_{\rho}$ and $\nu_{\rho}$ together with the result for Case~(2). 
Because $r \simge d_{\rho}(x,\infty)$, we get
\begin{equation*}
\mu_{\rho}(B_{\rho}(x,r))
\simeq \mu_{\rho}(B_{\rho}(\infty,r))
\simeq r^{\theta} \nu_{\rho}(B_{\rho}(\infty,r))
\simeq r^{\theta} \nu_{\rho}(B_{\rho}(x,r)).
\end{equation*}

Finally we consider Case~(4).
Because $d_{\rho}(y,\infty) \simeq \rho(|y|)(|y|+1)$ for every $y \in \Omega$ and $\rho(t)(t+1)$ is quasidecreasing, see~\cite[Remark~2.1]{KRST}, we get that $d_{\rho}(z,y) \leq d_{\rho}(z,\infty) + d_{\rho}(\infty,y) \simeq \rho(|y|)(|y|+1)$ for every $z , y \in \Omega$ such that $|z| \geq |y|$. In particular we get that $d_{\rho}(z,y) \simle \rho(1)$, which means that $\diam_{\rho}(\Omega) \simle \rho(1)$.
Also for any $y \in \partial \Omega$ such that $0<|y| \leq 1$, we have $d_{\rho}(y,\infty) \simeq \rho(|y|)(|y|+1) \simeq \rho(1)$.
Therefore $\rho(1) \simle \diam_{\rho}(\partial_{\rho}\Omega) \leq \diam_{\rho}(\Omega) \simle \rho(1)$ and thus in Case~(4) we have $\diam_{\rho}(\Omega) \simeq \diam_{\rho} (\partial_{\rho} \Omega) \simeq 
\rho(1) \simeq r$.
Also we get from \eqref{condDreverse} and \eqref{measure-of-boundary} that
\begin{equation*}
\rho(1)^{p-\theta} \nu(B(b,2))
\simle \int_{\partial\Omega \setminus B(b,1)} \rho(|y|)^{p-\theta} d\nu(y)
\leq \nu_{\rho}(\partial_{\rho}\Omega)
\leq C_D C_A^{p-\theta} \rho(1)^{p-\theta} \nu(B(b,2))
\end{equation*}
and thus by the doubling property of $\nu$ we have $\nu_{\rho}(\partial_{\rho}\Omega) \simeq \rho(1)^{p-\theta} \nu(B(b,1))$.
On the other hand for any $0 < s \leq 1$
\begin{equation*}
\int_{\Omega \setminus B(b,s)} \rho(|y|)^p d \mu(y)
\leq C_C \rho(s)^p \mu(B(b,s+1))
\leq C_C C_A^p \rho(1)^p \mu(B(b,2)).
\end{equation*}
Thus by letting $s \to 0$ we have $\mu_{\rho}(\Omega) \leq C_C C_A^p \rho(1)^p \mu(B(b,2))$. From \eqref{condCreverse} we get
\begin{equation*}
\rho(1)^p \mu(B(b,2))
\simle \int_{\Omega \setminus B(b,1)} \rho(|y|)^p d \mu(y)
\leq \mu_{\rho}(\Omega).
\end{equation*}
Thus by the doubling property of $\mu$ we have $\mu_{\rho}(\Omega) \simeq \rho(1)^p \mu(B(b,1))$.
Therefore we get
\begin{align*}
\nu_{\rho}(B_{\rho}(x,r))
&\simeq \nu_{\rho}(B_{\rho}(x,\diam_{\rho}(\partial_{\rho} \Omega)))
\simeq \nu_{\rho}(\partial_{\rho}\Omega)
\simeq \rho(1)^{p-\theta} \nu(B(b,1))
\simeq \rho(1)^{-\theta} \rho(1)^p \mu(B(b,1))
\\
&\simeq r^{-\theta} \mu_{\rho}(\Omega)
\simeq r^{-\theta} \mu_{\rho}(B_{\rho}(x,\diam_{\rho}(\Omega)))
\simeq r^{-\theta} \mu_{\rho}(B_{\rho}(x,r)).
\end{align*}
Here we used the doubling properties of $\mu_{\rho}$ and $\nu_{\rho}$ and the codimensionality of $\mu$ and $\nu$.
\end{proof}

\section{Quasi-preservation of Besov energy}
\label{sect-Besov}

As in Section~2, by structural data we mean the constants $C_U$, $C_{\mu}$, $p$, $\lambda$, $C_P$, $\theta$, $C_{\theta}$, $\kappa$, $C_A$, $C_B$, $C_C$ and $C_D$.
Of these, $C_U$, $C_\mu$, $p$, $\lambda$, $C_P$, $\theta$, $C_\theta$ and $\kappa$ are structural data associated with the space $(\Om,d,\mu)$ and the measure $\nu$.
In this section we explore how Besov spaces are transformed under the sphericalization procedure considered in the present paper.
We now know that 
\begin{enumerate}
\item when $\nu$ is $\theta$-codimensional with respect to $\mu$, the transformed measure $\nu_\rho$ is $\theta$-codimensional with respect to the transformed measure $\mu_\rho$,
\item when $(\Om, d)$ is a uniform domain with boundary $\partial\Om$, the transformed domain $(\Om, d_\rho)$ is also a uniform domain with boundary $\partial_\rho\Om=\partial\Om\cup\{\infty\}$,
\item when $(\Om, d, \mu)$ is doubling and supports a $p$-Poincar\'e inequality, the transformed space $(\Om, d_\rho, \mu_\rho)$ also is doubling and supports a $p$-Poincar\'e inequality.
\end{enumerate}
Given the above three, we have the following proposition by Malý~\cite{Maly} (see also \cite{GibaraSh}).

\begin{proposition}\label{prop:ExtTrace}
There are bounded linear trace operators
\[
T_d:D^{1,p}(\Om,d,\mu)\to \HB^{1-\theta/p}_{p,p}(\partial\Om,d,\nu)
\]
and
\[
T_\rho:N^{1,p}(\Om,d_\rho, \mu_\rho)\to B^{1-\theta/p}_{p,p}(\partial_\rho\Om,d_\rho,\nu_\rho)
\]
and bounded linear extension operators
\[
E_d: \HB^{1-\theta/p}_{p,p}(\partial\Om,d,\nu)\to D^{1,p}(\Om,d,\mu)
\]
and
\[
E_\rho:B^{1-\theta/p}_{p,p}(\partial_\rho\Om,d_\rho,\nu_\rho)\to N^{1,p}(\Om,d_\rho,\mu_\rho)
\]
such that $T_d\circ E_d$ is the identity map on $\HB^{1-\theta/p}_{p,p}(\partial\Om,d,\nu)$ while $T_\rho\circ E_\rho$ is the identity map on $B^{1-\theta/p}_{p,p}(\partial_\rho\Om,d_\rho,\nu_\rho)$.

For $u\in D^{1,p}(\Om, d,\mu)$ and for $\nu$-a.e.~$\xi\in \partial\Om$, we have
\[
T_du(\xi)=\lim_{r\to 0^+}\vint_{B(\xi,r)}u\, d\mu=\lim_{r\to 0^+}\vint_{B_\rho(\xi,r)} u\, d\mu_\rho=T_\rho u(\xi),
\]
and so we can unambiguously denote $T_du=T_\rho u\vert_{\partial\Om}$ by $Tu$.
For each $\xi\in\partial\Om$ where the above Lebesgue value of $u$ exists, we set $T_d u(\xi)$ to be that value; a similar statement holds for $T_\rho u$.
This determines the value of $Tu$ more precisely than $\nu$-almost everywhere.

Moreover, for each $u\in N^{1,p}(\Om,d_\rho,\mu_\rho)$ and $v\in B^{1-\theta/p}_{p,p}(\partial_\rho\Om, d_\rho,\nu_\rho)$ we have the following energy bounds:
\begin{align*}
\|T_\rho u\|_{B^{1-\theta/p}_{p,p}(\partial_\rho\Om, d_\rho,\nu_\rho)}&\lesssim \|g_u\|_{L^p(\Om,\mu_\rho)},\\
\|g_{E_\rho v}\|_{L^p(\Om,\mu_\rho)} &\lesssim \|v\|_{B^{1-\theta/p}_{p,p}(\partial_\rho\Om, d_\rho,\nu_\rho)}.
\end{align*}
Here the minimal $p$-weak upper gradients are with respect to the space $(\Omega,d_{\rho},\mu_{\rho})$.
The comparison constants associated with the above estimates depend only on the structural data. 
The bounds for the linear operators $T_d$ and $E_d$ depend solely on the structural data associated with $(\Om,d,\mu)$ and $\nu$.
\end{proposition}

Recall that subsets of $\partial\Om$ that are $\nu_\rho$-measure zero are precisely the sets of $\nu$-measure zero.
Moreover, $\nu_\rho(\{\infty\})=0$.
However, it is possible for $\{\infty\}$ to have positive $p$-capacity in the space $(\overline{\Om}^{d_\rho}, d_\rho,\mu_\rho)$. When $\{\infty\}$ has positive $p$-capacity, necessarily $\infty$ is a $d_\rho$-Lebesgue point of each $u\in N^{1,p}(\Om,d_\rho,\mu_\rho)$ (see~\cite[Theorem~9.2.8]{HKST}) and so $T_\rho u(\infty)$ is set to be equal to the corresponding Lebesgue value, but $T_du(\infty)$ does not make sense. Hence, by $Tu$ we only mean $T_du=T_\rho u\vert_{\partial\Om}$.


We also know from \cite[Proposition 5.4]{KRST} that if $u \in D^{1,p}(\Omega,d,\mu)$, then $u \in D^{1,p}(\Omega,d_{\rho},\mu_{\rho})$ with the property that
\begin{equation}\label{eq:Energy-isom}
\int_\Om g_u^p\, d\mu=\int_\Om \widehat g_{u}^p\, d\mu_\rho,
\end{equation}
where $g_u$ is an upper gradient of $u$ with respect to the space $(\Om, d,\mu)$ while $\widehat g_{u}$ is an upper gradient of $u$ with respect to the space $(\Om,d_\rho,\mu_\rho)$ with
\begin{equation}\label{eq:minimal-weak-transform}
\widehat g_{u}=\frac{1}{\rho}\, g_u.
\end{equation}
Since $p$-weak upper gradients that are $p$-integrable are approximable in $L^p$-spaces by upper gradients, the identity~\eqref{eq:Energy-isom} also holds for $p$-weak upper gradients with respect to $(\Om, d,\mu)$ and $(\Om,d_\rho,\mu_\rho)$ respectively. Such an approximation gives the corresponding relationship~\eqref{eq:minimal-weak-transform} between the minimal $p$-weak upper gradient of $u$ with respect to $(\Om, d,\mu)$ and the minimal $p$-weak upper gradient of $u$ with respect to $(\Om,d_\rho,\mu_\rho)$.
From~\cite[Theorem~9.1.2]{HKST}, we know that the support of $p$-Poincar\'e inequality implies the stronger inequality where the left-hand side also is an $L^p$-variance; hence it follows that functions in $D^{1,p}$-spaces over a bounded metric measure space equipped with a doubling measure supporting a $p$-Poincar\'e inequality are in the global $L^p$-class.
Since $(\Om,d_\rho,\mu_\rho)$ is bounded and supports a $p$-Poincar\'e inequality, we have $D^{1,p}(\Omega,d_{\rho},\mu_{\rho}) = N^{1,p}(\Omega,d_{\rho},\mu_{\rho})$.
Next we show that our sphericalization preserves the Besov energy.

\begin{theorem}\label{thm:preserve-Besov}
As sets, we have that $\HB^{1-\theta/p}_{p,p}(\partial\Om, d,\nu)=B^{1-\theta/p}_{p,p}(\partial_\rho\Om, d_\rho,\nu_\rho)$, and moreover, there is a constant $C\ge 1$ such that for each $u\in \HB^{1-\theta/p}_{p,p}(\partial\Om, d,\nu)$ we have that
\[
\frac{1}{C}\, \|u\|_{B^{1-\theta/p}_{p,p}(\partial\Om, d,\nu)}\le \|u\|_{B^{1-\theta/p}_{p,p}(\partial_\rho\Om, d_\rho,\nu_\rho)}
\le C\, \|u\|_{B^{1-\theta/p}_{p,p}(\partial\Om, d,\nu)}.
\]
The constant $C$ depends only on the structural data.
\end{theorem}

\begin{proof}
Applying Proposition~\ref{prop:ExtTrace} to $u\in \HB^{1-\theta/p}_{p,p}(\partial\Om, d,\nu)$ and~\eqref{eq:Energy-isom}, we note that $E_d u$ belongs in $D^{1,p}(\Om, d, \mu)$ with
\[
\int_\Om \widehat g_{E_d(u)}^p\, d\mu_\rho=\int_\Om g_{E_d(u)}^p\, d\mu\le C\, \|u\|_{B^{1-\theta/p}_{p,p}(\partial\Om, d,\nu)}^p.
\]
By Proposition~\ref{prop:ExtTrace}, we know that $T_\rho\circ E_d(u)=T_d\circ E_d(u)=u$.
Applying Proposition~\ref{prop:ExtTrace} again, we see that
\[
\|u\|_{B^{1-\theta/p}_{p,p}(\partial_\rho \Om, d_\rho, \nu_\rho)}^p\le C \int_\Om \widehat g_{E_d(u)}^p\, d\mu_\rho
   \le C\, \|u\|_{B^{1-\theta/p}_{p,p}(\partial\Om, d,\nu)}^p.
\]
Once we see that $u\in L^p(\partial_\rho\Om,\nu_\rho)$, which follows from \cite[Lemma 2.2]{CapognaKlineKorteShanmugalingamSnipes}, we get that $\HB^{1-\theta/p}_{p,p}(\partial\Om, d,\nu)$ is a subset of $B^{1-\theta/p}_{p,p}(\partial_\rho\Om, d_\rho,\nu_\rho)$.

Now, if $v\in B^{1-\theta/p}_{p,p}(\partial_\rho \Om, d_\rho, \nu_\rho)$, then again noting that $T_d \circ E_\rho(v)=v$, the boundedness of the operators $T_d$ and $E_\rho$ as described in Proposition~\ref{prop:ExtTrace}, together with \eqref{eq:Energy-isom}, we also see that
\[
\|v\|_{B^{1-\theta/p}_{p,p}(\partial\Om, d,\nu)}^p\le C\, \|v\|_{B^{1-\theta/p}_{p,p}(\partial_\rho \Om, d_\rho, \nu_\rho)}^p,
\]
completing the proof.
\end{proof}

\section{Notions of Dirichlet boundary condition}
\label{sect-dirichlet-notions}

A standard problem in partial differential equations and potential theory is the following. We have a domain $U$ and a differential operator $L$ or an energy operator $I$ on a space of functions on the domain. The associated Dirichlet problem is to find a function $u$ in this space that either satisfies $Lu=0$ in $U$ (in some sense) or minimizes the energy $I$, that is, $I(u)\le I(v)$ for each $v$ in the same function space with both $u$ and $v$ satisfying a fixed Dirichlet boundary condition on $\partial U$. In the setting of Riemannian manifolds, one useful class of operators is the class of $p$-Laplace-Beltrami operators $\Delta_p$, with a Sobolev function $u$ satisfying $\Delta_pu=0$ in a domain $U$ in that manifold if
\[
\int_U|\nabla u|^{p-2}\langle\nabla u, \nabla v\rangle\, d\mu=0,
\]
whenever $v$ is in a Sobolev class with compact support contained in $U$. A Sobolev function $u$ satisfies the above equation if and only if for every Sobolev function $v$ with compact support we have
\[
I(u):=\int_U|\nabla u|^p\, d\mu\le \int_U|\nabla (u+v)|^p\, d\mu=I(u+v).
\]
In the present paper, we are interested in the following Dirichlet problem. 

\begin{definition}\label{def:DProb}
Given a function $f\in HB^{1-\theta/p}_{p,p}(\partial\Om,d,\nu)$, we say that a function $u\in D^{1,p}(\Om,d,\mu)$ solves the Dirichlet problem for $p$-harmonic functions with boundary value $f$ if
\begin{enumerate}
\item $u$ is $p$-harmonic in $\Om$, that is, $u$ is continuous in $\Om$ with
\[
\int_\Om g_u^p\, d\mu\le \int_\Om g_{u+v}^p\, d\mu
\]
whenever $v\in D^{1,p}(\Om,d,\mu)$ with compact support in $\Om$,
\item $T_du=f$ $\nu$-a.e.~on $\partial\Om$.
\end{enumerate}
\end{definition}

If there is a function $u\in D^{1,p}(\Om,d,\mu)$ that satisfies all of the conditions in the above definition except for continuity on $\Om$, then a modification of $u$ on a set of $p$-capacity zero yields a locally H\"older continuous function satisfying all of the conditions in the above definition, see~\cite{KinSh}.

In the formulation of what it means for a function $u$ to be $p$-harmonic in $\Om$, a related notion is that of the so-called Cheeger $p$-harmonicity, where the upper gradient energy $\int_\Om g_u^p\, d\mu$ is replaced by the integral $\int_\Om|\nabla u|^p\, d\mu$ where $\nabla u$ is a Cheeger differential of $u$. Such a differential exists as a measurable linear differential operator on $N^{1,p}(\Om)$ when $(\Om, d,\mu)$ is doubling and supports a $p$-Poincar\'e inequality, see for instance~\cite{HKST}. The advantage of Cheeger $p$-harmonicity is that there is a corresponding weak notion of a differential equation that is the Euler-Lagrange equation corresponding to the minimization of the related integral; however, there are many possible choices of Cheeger differential structures, and hence, unlike the upper gradients, they are not as intrinsic to the metric measure space. However, the results obtained in this section also hold for Cheeger $p$-harmonic solutions as well via mutatis mutandis changes in the proofs.

In formulating the Dirichlet problem we can deviate from the above traditional formulation by asking for $T_d v=0$ $\nu$-a.e.~on $\partial\Om$ rather than that $v$ has compact support in $\Om$. It is reasonable to ask whether if, instead of requiring compact support of $v$ in $\Om$ we require $T_dv=0$ $\nu$-a.e.~on $\partial\Om$, we end up with a different solution to the problem. In this section we will demonstrate that the two formulations agree.
This demonstration is done via Theorem~\ref{lem-zeroboundaryvalues} below, by proving that $D_0^{1,p}(\Om,d,\mu)=N^{1,p}_0(\Om,d_\rho,\mu_\rho)$ as sets, for then the density of compactly supported Lipschitz functions in $N_0^{1,p}(\Om,d_\rho,\mu_\rho)$ as demonstrated in~\cite{Shan-Illinois} together with \eqref{eq:Energy-isom} completes this argument.
Note that compactness in $\Om$ with respect to the topology generated by the metric $d_\rho$ is the same as the compactness with respect to the topology generated by the original metric $d$, since the topology on $\Om$ is preserved, see~\cite[Proposition~3.2]{KRST}. 
The two problems do not coincide in the setting considered in~\cite{GibaraKorteSh}, where $\Om$ was unbounded but $\partial\Om$ was bounded, in which case the two notions differ when $\overline \Omega^d$ is also $p$-hyperbolic. Recall that $\partial\Om$ is unbounded in the present paper.

\begin{remark}\label{rem:Newt-closure}
As $(\Om,d_\rho,\mu_\rho)$ is doubling and supports a $p$-Poincar\'e inequality, we know from~\cite[Proposition~7.1]{AikawaShan-Carleson} that $N^{1,p}(\Om,d_\rho,\mu_\rho)=N^{1,p}(\overline{\Om}^{d_{\rho}}, d_\rho, \mu_\rho)$.
In what follows, $N^{1,p}_0(\Om,d_\rho,\mu_\rho)$ denotes the collection of all $v\in N^{1,p}(\Om,d_\rho,\mu_\rho)=N^{1,p}(\overline{\Om}^{d_{\rho}}, d_\rho, \mu_\rho)$ for which
\[
\cp_p(\{x\in\partial_\rho\Om\, :\, v(x)\ne 0\})=0,
\]
where $\cp_p$ is the $p$-capacity as given in Definition~\ref{def:capacities}.
Here the $p$-capacity is with respect to the space $(\overline\Omega^{d_{\rho}},d_{\rho},\mu_{\rho})$.
The transformed space $(\Om,d_\rho,\mu_\rho)$ is locally compact and $(\overline\Om^{d_{\rho}},d_\rho,\mu_\rho)$ is compact, doubling and supports a $p$-Poincar\'e inequality by \cite[Proposition~7.1]{AikawaShan-Carleson}. Hence, from~\cite[Theorem~4.8]{Shan-Illinois}, we know that $N^{1,p}_0(\Om,d_\rho,\mu_\rho)$ is also obtained as the $N^{1,p}$-norm-closure of Lipschitz functions in $N^{1,p}(\Om,d_\rho,\mu_\rho)$ with compact support contained in $\Om$.
Note that functions that are compactly supported in $(\Om,d_\rho)$ are also compactly supported in $(\Om, d)$.
\end{remark}

\begin{lemma}\label{lem:zero} 
Suppose that $u\in D^{1,p}(\Omega,d,\mu)$ such that $T_du=0$ $\nu$-a.e. on $\partial\Omega$.
With $W$ denoting the set of points $x\in\partial_\rho\Omega$ for which $\limsup_{r\to 0^+} |u_{B_\rho(x,r)}|>0$, we have that $\cp_p(W)=0$ when $\cp_p$ is measured with respect to the space $(\overline{\Om}^{d_\rho},d_\rho,\mu_\rho)$.
\end{lemma}

\begin{proof} 
In this proof, unless otherwise stated, the balls are with respect to the space $(\overline{\Om}^{d_\rho},d_\rho,\mu_\rho)$.
We first assume that $0\le u\le 1$.
We get from \cite[Theorem 5.4]{KRST} that $u \in N^{1,p}(\Omega,d_{\rho},\mu_{\rho})$, and by the above discussion, we also have that $u\in N^{1,p}(\overline{\Om}^{d_\rho},d_{\rho},\mu_{\rho}$).

Fix $0<\delta\le 1$ and let $A$ be the set of all points $x\in \partial_{\rho} \Omega$ such that $\limsup_{r\rightarrow0^+}u_{B_\rho(x,r)}>\delta$.
It is enough to prove that $\cp_p(A)=0$ for each $0<\delta\le 1$.
Note that in the rest of the proof, the comparison constants are allowed to depend on $\delta$.

Fix $0<R<1$. 
Using the $5B$-covering theorem (see for instance~\cite[page~60]{HKST}), we cover $A$ by balls $B_\rho(x_i, 10 \lambda r_i)$ such that the balls $B_\rho(x_i,2 \lambda r_i)$ are disjoint, $x_i \in A$, $u_{B_\rho(x_i,r_i)}> \delta$ and $r_i<R$.
Note that as $u\le 1$,
\begin{align*}
\delta &< \vint_{B_\rho(x_i,r_i)}u\, d\mu_\rho=\frac{1}{\mu_\rho(B_\rho(x_i,r_i))}\left[\int_{B_\rho(x_i,r_i)\cap\{u\le \delta/2\}} u\, d\mu_\rho
   +\int_{B_\rho(x_i,r_i)\cap\{u> \delta/2\}} u\, d\mu_\rho\right]\\
  & \le \frac{\tfrac{\delta}{2}\, \mu_\rho(B_\rho(x_i,r_i)\cap\{u\le \delta/2\})}{\mu_\rho(B_\rho(x_i,r_i))}
     +\frac{\mu_\rho(B_\rho(x_i,r_i)\cap\{u> \delta/2\})}{\mu_\rho(B_\rho(x_i,r_i))}\\
   &\le \frac{\delta}{2}+\frac{\mu_\rho(B_\rho(x_i,r_i)\cap\{u> \delta/2\})}{\mu_{\rho}(B_\rho(x_i,r_i))}.
\end{align*}
It follows that $\mu_\rho(\{u>\delta/2\}\cap B_\rho(x_i,r_i)) > \frac{\delta}{2}\mu_\rho(B_\rho(x_i,r_i))$.
Let $E_i=\{u>\delta/2\}\cap B_\rho(x_i,r_i)$ and $F_i=\{x\in \partial_{\rho} \Omega\cap B_\rho(x_i,r_i)\,:\, u(x)=T_{\rho} u(x)=0\}$.
Note that $T_{\rho} u = u$ $p$-capacity a.e. on $\partial_{\rho} \Omega$ and thus $\nu_{\rho}$-a.e. on $\partial_{\rho} \Omega$, see Lemma \ref{lem:capacity2nu}.
As $T_{\rho} u=0$ $\nu_\rho$-a.e. on $\partial_\rho \Omega$ (see Proposition~\ref{prop:ExtTrace}), we have $\nu_\rho(F_i) = \nu_\rho(B_{\rho}(x_i,r_i))$.
As $\mu_\rho(E_i)\simeq\mu_\rho(B_\rho(x_i,r_i))$, we have that $\mathcal H_{r_i}^{-\theta}(E_i)\simeq\frac{\mu_\rho(B_\rho(x_i,r_i))}{r_i^\theta}$ by Lemma~\ref{lemma:measure content}. Here the comparison constant depends solely on $\delta/2$.

We next estimate the $\theta$-codimensional Hausdorff content of $F_i$. Let $B_\rho(z_j,s_j)$ be a covering of $F_i$ with balls such that $s_j \leq r_i$ and $B_\rho(z_j,s_j)\cap F_i$ is non-empty. Then there exists $y_j \in B_\rho(z_j,s_j) \cap F_i$. Thus $\mu_\rho(B_\rho(y_j,2s_j)) \simeq \mu_\rho(B_{\rho}(z_j,s_j))$ and $B_\rho(z_j,s_j) \subset B_\rho(y_j,2s_j)$.
Then by the $\theta$-codimensionality of $\nu_\rho$ we have
\begin{align*}
\sum_j \frac{\mu_\rho(B_\rho(z_j,s_j))}{s_j^\theta}
\simeq \sum_j  
\frac{\mu_\rho(B_\rho(y_j,2s_j))}{(2s_j)^\theta}
&\simeq \sum_j 
\nu_\rho(B_\rho(y_j,2s_j))\\
&\ge\nu_\rho(F_i)
=\nu_\rho(B_\rho(x_i,r_i))
\simeq \frac{\mu_\rho(B_\rho(x_i,r_i))}{r_i^\theta},
\end{align*}
and consequently also $\mathcal H_{r_i}^{-\theta}(F_i)\gtrsim\frac{\mu_\rho(B_\rho(x_i,r_i))}{r_i^\theta}$.

As $\frac{2}{\delta}u=0$ on $F_i$ and $\frac{2}{\delta}u\ge 1 $ on $E_i$, we have by Proposition~\ref{prop:capest},
\[
\int_{B_\rho(x_i, 2 \lambda r_i)}\widehat g_u^pd\mu_\rho 
\simeq \int_{B_\rho(x_i, 2 \lambda r_i)} \widehat g_{\frac{2}{\delta}u}^p d\mu_\rho
\simge \mu_\rho(B_\rho(x_i,r_i))/r_i^p.
\]
Note that $\frac{2}{\delta}u \in N^{1,p}(B_\rho(x_i,2 \lambda r_i),d_{\rho},\mu_\rho)$.
From~\cite[Lemma 9.2.3]{HKST} 
or by using the function
\[
v(x)=\left(1-\frac{\dist_\rho(x,B_\rho(x_i,10 \lambda r_i))}{10\lambda r_i}\right)_+,
\]
we know that $\cp_p(B_\rho(x_i,10 \lambda r_i))\lesssim \mu_\rho(B_\rho(x_i,10 \lambda r_i))/(10 \lambda r_i)^p$.
Consequently,
\begin{align*}
\cp_p(A)
&\le \sum_i\cp_p(B_\rho(x_i, 10 \lambda r_i))\
\lesssim \sum_i\frac{\mu_\rho(B_\rho(x_i, 10 \lambda r_i))}{(10 \lambda r_i)^p}
\simeq \sum_i\frac{\mu_\rho(B_\rho(x_i,r_i))}{r_i^p}
\\
&\simle \sum_i  \int_{B_\rho(x_i, 2 \lambda r_i)}\widehat g_u^pd\mu_\rho
= \int_{\cup_iB_\rho(x_i, 2 \lambda r_i)}\widehat g_u^pd\mu_\rho.
\end{align*}
As $\bigcup_i B_\rho(x_i, 2 \lambda r_i)\subset \{x\in\overline\Omega^{d_\rho}: \dist_\rho(x,\partial_\rho\Omega)< 2 \lambda R\}$, by using the dominated convergence theorem and the fact that $\widehat g_u\in L^p(\Om,\mu_{\rho})$, we see that the last integral converges to $0$ as $R\rightarrow 0$. Thus $\cp_p(A)=0$ as desired.

Finally we remove the constraint that $0\le u\le 1$.
Assume that $u \geq 0$. Define $v := \min\{u,1\}$.
Then $v \in D^{1,p}(\Omega,d,\mu)$ and $T_d v = 0$ $\nu$-a.e. on $\partial \Omega$. Therefore $v$ also satisfies the conditions of the lemma and thus for $A := \{ x \in \partial_{\rho} \Omega : \limsup_{r \to 0^+} v_{B_{\rho}(x,r)} > 0 \}$ we have $\cp_p(A)=0$.
Now let $E \subset \partial_{\rho} \Omega$ be the set of points that are not Lebesgue points of $u$. Similarly let $F \subset \partial_{\rho} \Omega$ be the set of points that are not Lebesgue points of $v$. According to \cite[Theorem 9.2.8]{HKST} for the case $p>1$ and \cite[Theorem~4.1]{KKST} for the case $p=1$, we have $\cp_p(E) = \cp_p(F) = 0$.
Finally if $x \in \partial_{\rho} \Omega \setminus (A \cup E \cup F)$, then $\min\{u(x),1\}=v(x) = \lim_{r \to 0^+} v_{B_\rho(x,r)} = 0$ and therefore $\lim_{r \to 0^+} u_{B_\rho(x,r)} = u(x) = 0$. The result follows because $\cp_p(A \cup E \cup F)=0$.

We get the result for general $u$ by considering the negative and positive parts separately.
\end{proof}
 
\begin{theorem} \label{lem-zeroboundaryvalues}
Let $D^{1,p}_0(\Om,d,\mu)$ be the collection of all $v\in D^{1,p}(\Om,d,\mu)$ with $T_d v=0$ $\nu$-a.e.~on $\partial\Om$. 
Then
\[
D^{1,p}_0(\Om,d,\mu)=N^{1,p}_0(\Om,d_\rho,\mu_\rho).
\]
\end{theorem}

\begin{proof}

Suppose first that $v \in D_0^{1,p}(\Omega,d,\mu)$.
We get from \cite{KRST} that $v \in N^{1,p}(\Omega,d_{\rho},\mu_{\rho})$.
Note that $(\Omega,d_{\rho},\mu_{\rho})$ supports a $p$-Poincar\'e inequality by \cite{KRST}, and so, by~\cite[Proposition 7.1]{AikawaShan-Carleson}, we get that $v$ has an extension to $\overline{\Om}^{d_\rho}$, also denoted by $v$, such that $v \in N^{1,p}(\overline \Omega^{d_{\rho}},d_{\rho},\mu_{\rho})$, and the space $(\overline \Omega^{d_{\rho}},d_{\rho},\mu_{\rho})$ supports a $p$-Poincar\'e inequality, see Remark~\ref{rem:Newt-closure}.
According to \cite[Theorem 9.2.8]{HKST} for the case $p>1$ and \cite[Theorem~4.1]{KKST} for the case $p=1$, $p$-capacity almost every point in $(\overline \Omega^{d_{\rho}},d_{\rho},\mu_{\rho})$ is a Lebesgue point of $v$. Thus we may assume that $v(x)=0$ whenever $x \in \overline \Omega^{d_{\rho}}$ such that $\lim_{r \to 0^+} \vint_{B_{\rho}(x,r)} v(y) d \mu_{\rho} (y) = 0$.
By applying Lemma \ref{lem:zero} we then get
\begin{equation*}
\cp_p(\{x \in \partial_{\rho} \Omega : v(x) \neq 0 \})
\leq \cp_p \left( \left\{x \in \partial_{\rho} \Omega : \lim_{r \to 0^+} \vint_{B_{\rho}(x,r)} v(y) d \mu_{\rho} (y) \neq 0 \right\} \right) = 0,
\end{equation*}
and so $v\in N^{1,p}_0(\Om,d_\rho,\mu_\rho)$.

Now let $v \in N_0^{1,p}(\Omega,d_{\rho},\mu_{\rho}) \subset N^{1,p}(\Omega,d_{\rho},\mu_{\rho}) = D^{1,p}(\Omega,d,\mu)$. Then, as seen in Remark~\ref{rem:Newt-closure}, there exist functions $v_i \in N^{1,p}(\Omega,d_{\rho},\mu_{\rho})$ that are compactly supported and $v_i$ converges to $v$ in the $N^{1,p}(\Omega,d_{\rho},\mu_{\rho})$ norm. By Proposition~\ref{prop:ExtTrace} we then have $T_{\rho} v_i(x) = \lim_{r \to 0^+} \vint_{B_{\rho}(x,r)} v_i(y) d \mu_{\rho}(y) = 0$ for $\nu_{\rho}$-a.e. $x \in \partial_{\rho} \Omega$. Therefore by Proposition \ref{prop:ExtTrace},
\begin{equation*}
\| T_{\rho} v \|_{L^p(\partial_{\rho} \Omega , \nu_{\rho})}
= \| T_{\rho} (v-v_i) \|_{L^p(\partial_{\rho} \Omega , \nu_{\rho})}
\leq C \| v - v_i \|_{N^{1,p}(\Omega,d_{\rho},\mu_{\rho})}
\to 0\text{ as }i\to\infty.
\end{equation*}
Thus $T_{\rho} v(x) = 0$ for $\nu_{\rho}$-almost every $x \in \partial_{\rho} \Omega$ and thus by 
Proposition~\ref{prop:ExtTrace} we have $T_d v(x) = 0$ for $\nu$-almost every $x \in \partial \Omega$, which is what we wanted to show.
\end{proof}

\section{Solving the Dirichlet problem in unbounded uniform domains: proof of Theorem~\ref{thm:Main}}\label{Sec:MainProof}

In this section we consider the problem of finding a function $u_f\in D^{1,p}(\Om,d,\mu)$ for each $f\in HB^{1-\theta/p}_{p,p}(\partial\Om,d,\nu)$ such that $u_f$ is $p$-harmonic in $\Om$ and $Tu_{f}=f$ $\nu$-a.e.~on $\partial\Om$ when $\partial\Om$ is unbounded. Here, we say that a function $u\in D^{1,p}(\Om,d,\mu)$ is $p$-harmonic in $\Om$ if for each $v\in D^{1,p}(\Om,d,\mu)$ with $Tv=0$ $\nu$-a.e.~on $\partial\Om$, we have
\[
\int_\Om g_u^p\, d\mu\le \int_\Om g_{u+v}^p\, d\mu.
\]
The case $p=1$ is associated more with the theory of functions of bounded variation, and the case $p=\infty$ is associated more with the theory of absolutely minimal Lipschitz extensions, both of which are of a different nature from the case $1<p<\infty$. In this paper we focus on the case $1<p<\infty$.

The goal of this section is to prove the main theorem of this paper, Theorem~\ref{thm:Main}. For the convenience of the reader, we repeat the statement of the theorem next.

\begin{theorem}
Let $p > 1$.
For each $f\in HB^{1-\theta/p}_{p,p}(\partial\Om,d,\nu)$ there is a unique function $u_f\in D^{1,p}(\Om,d,\mu)$ that is $p$-harmonic in $(\Om,d,\mu)$ and $T u_f = f$ $\nu$-a.e.~on $\partial\Om$.
\end{theorem}

We point out here that the proof below of the uniqueness claim of the theorem follows along the lines of the proof of a similar uniqueness result in~\cite{Chee}. As the structure considered in~\cite{Chee} is not precisely that considered in this paper, we provide the proof of that segment as well for the convenience of the reader.

\begin{proof}
Fixing $f\in HB^{1-\theta/p}_{p,p}(\partial \Om,d,\nu)$, we use Theorem~\ref{thm:preserve-Besov} to see that also $f\in B^{1-\theta/p}_{p,p}(\partial_{\rho} \Om,d_\rho,\nu_\rho)$.
We get from Lemma \ref{lem:capacity2nu} that $\text{Cap}_p(\partial_\rho\Om)>0$ where the $p$-capacity is with respect to the space $(\overline{\Om}^{d_\rho}, d_\rho,\mu_\rho)$.
Hence by Maz'ya's inequality as in~\cite[Theorem~5.53]{bjornbjorn} (see~\cite{Maz} for the original proof in the Euclidean setting), we have that
\begin{equation}\label{eq:Mazya}
\int_\Om|v|^p\, d\mu_\rho\le C\, \int_\Om \widehat g_v^p\, d\mu_\rho=C\, \int_\Om g_v^p\, d\mu
\end{equation}
whenever $v\in N^{1,p}(\Om,d_\rho,\mu_\rho)$ with $T_{\rho} v=0$ $\nu_\rho$-a.e.~on $\partial_\rho\Om$.
Here $\widehat g_v$ is the minimal $p$-weak upper gradient of $v$ in $(\Om,d_\rho,\mu_{\rho})$, and we used \eqref{eq:Energy-isom}.
The comparison constant depends on $\mu_{\rho}(\Omega)$, $\diam_{\rho}(\Omega)$, $p$, $\cp_p(\partial_{\rho} \Omega)$ and $C_P$.

We now proceed to solve the problem of finding a function $u_f\in N^{1,p}(\Om,d_\rho,\mu_\rho)$ that is $p$-harmonic in $\Om$ with respect to the metric $d_\rho$ and with trace $T_\rho(u_f)=f$ $\nu$-a.e.~on $\partial\Om$ as follows. Let
\[
\Delta:=\inf\bigg\lbrace \int_\Om \widehat g_u^p\, d\mu_\rho\, :\, u\in N^{1,p}(\Om,d_\rho,\mu_\rho)\text{ with }T_\rho u=f \,\nu\text{-a.e.~on } \partial\Om \bigg\rbrace.
\]
Since $f\in B^{1-\theta/p}_{p,p}(\partial_\rho\Om, d_\rho,\nu_\rho)$, by Proposition~\ref{prop:ExtTrace} we know that $\Delta\le \int_\Om\widehat g_{E_\rho f}^p\, d\mu_\rho<\infty$. Thus we can find a sequence $(u_k)_k$ from $N^{1,p}(\Om,d_\rho,\mu_\rho)$ with $T_\rho u_k=f$ $\nu$-a.e.~on $\partial\Om$ such that 
\[
\lim_{k \to \infty} \int_\Om \widehat g_{u_k}^p\, d\mu_\rho=\Delta.
\]

If $\Delta=0$, then the functions $u_k$ have minimal $p$-weak upper gradients $\widehat{g_k}\in L^p(\Om,\mu_\rho)$ such that $\widehat{g_k}\to 0$ in $L^p(\Om,\mu_\rho)$. Moreover, $T_\rho u_k=f$ $\nu$-a.e.~on $\partial\Om$ for each $k$, and so $T_\rho(u_k-u_m)=0$ $\nu$-a.e.~on $\partial\Om$ for each $k,m\in\N$, and $\widehat g_{u_k-u_m}\le \widehat{g_k}+\widehat{g_m}$.
Applying~\eqref{eq:Mazya} to $u_k-u_m$ we also get $\int_\Om|u_k-u_m|^p\, d\mu_\rho\le C\, \int_\Om[\widehat{g_k}^p+\widehat{g_m}^p]\, d\mu_{\rho}$. It follows that $(u_k)_k$ is a Cauchy sequence in $N^{1,p}(\Om,d_\rho,\mu_\rho)$ that converges to a constant function (the limit function has zero function as a $p$-weak upper gradient, and so by the Poincar\'e inequality, we know that the limit function is also constant). 
As $T_\rho$ is a bounded linear operator, it follows that $f=T_\rho u_k\to c$ where $c$ is the constant function arising as the limit function. Therefore $f$ is constant, and naturally the limit constant function is the unique solution to the corresponding Dirichlet problem.

For the rest of the proof we assume that $\Delta>0$.
Without loss of generality we may also assume that $\int_\Om \widehat g_{u_k}^p\, d\mu_\rho\le 2\Delta$.
Moreover, by~\eqref{eq:Mazya} we know that
\begin{align*}
\|u_k\|_{L^p(\Om,\mu_\rho)}
&\le \|u_k-E_\rho f\|_{L^p(\Om,\mu_\rho)}+\|E_\rho f\|_{L^p(\Om,\mu_\rho)}\\
&\lesssim\, \|\widehat g_{u_k-E_\rho f}\|_{L^p(\Om,\mu_\rho)}+\, \|f\|_{B^{1-\theta/p}_{p,p}(\partial_{\rho} \Om,d_{\rho},\nu_{\rho})} 
+ \|f\|_{L^p(\partial_{\rho} \Omega , \nu_{\rho})}\\
&\leq \|\widehat g_{u_k}\|_{L^p(\Om,\mu_\rho)}+\|\widehat g_{E_\rho f}\|_{L^p(\Om,\mu_\rho)}
+\|f\|_{B^{1-\theta/p}_{p,p}(\partial_{\rho} \Om,d_{\rho},\nu_{\rho})} + \|f\|_{L^p(\partial_{\rho} \Omega , \nu_{\rho})}\\
&\lesssim 2^{1/p}\Delta^{1/p}
+2\|f\|_{B^{1-\theta/p}_{p,p}(\partial_{\rho} \Om,d_{\rho},\nu_{\rho})} + \|f\|_{L^p(\partial_{\rho} \Omega , \nu_{\rho})}
<\infty.
\end{align*}
In the above estimates we also used Proposition \ref{prop:ExtTrace}. 
It follows that $(u_k)_k$ is a bounded sequence in $N^{1,p}(\Om,d_\rho,\mu_\rho)= N^{1,p}(\overline\Om^{d_{\rho}},d_\rho,\mu_\rho)$, and by the reflexivity of $N^{1,p}(\overline\Om^{d_{\rho}},d_\rho,\mu_\rho)$ (which follows from the facts that $\mu_\rho$ is doubling and that $(\Om,d_\rho, \mu_\rho)$ supports a $p$-Poincar\'e inequality, see~\cite[Theorem~13.5.7]{HKST}) together with Mazur's lemma we can find a sequence $(v_n)_n$ of convex combinations of $u_k$ with 
\[
v_n=\sum_{k=n}^{N(n)}\lambda_{n,k}\, u_k, \text{ for a choice of nonnegative }\lambda_{n,k}\text{ with }
\sum_{k=n}^{N(n)}\lambda_{n,k}=1
\]
and a function $u_f\in N^{1,p}(\Om,d_\rho,\mu_\rho)$ such that $v_n\to u_f$ in $N^{1,p}(\Om,d_\rho,\mu_\rho)$ as $n\to \infty$.
Note that as $T_\rho$ is linear, we have $T_\rho(v_n)=\sum_{k=n}^{N(n)}\lambda_{n,k}\, T_\rho(u_k)=f$, and as $T_\rho$ is a bounded linear operator and hence is continuous, we also have that $T_\rho(u_f)=f$. It follows that
\[
\Delta\le \int_\Om \widehat g_{u_f}^p\, d\mu_\rho=\lim_{n\rightarrow\infty}\int_\Om\widehat g_{v_n}^p\, d\mu_\rho.
\]
On the other hand,
\[
\left(\int_\Om\widehat g_{v_n}^p\, d\mu_\rho\right)^{1/p}
\le \sum_{k=n}^{N(n)}\lambda_{n,k}\, \left(\int_\Om\widehat g_{u_k}^p\, d\mu_\rho\right)^{1/p}\to\Delta^{1/p}\text{ as }n\to\infty.
\]
Thus it follows that
\[
\int_\Om \widehat g_{u_f}^p\, d\mu_\rho=\Delta.
\]

Suppose now that $v\in N^{1,p}(\Om,d_{\rho},\mu_{\rho})$ such that $T_\rho v=0$. Then $T_\rho(u_f+v)=f$, and so
\[
\int_\Om \widehat g_{u_f}^p\, d\mu_\rho=\Delta\le \int_\Om\widehat g_{u_f+v}^p\, d\mu_\rho,
\]
that is, $u_f$ is $p$-harmonic in $(\Om,d_\rho,\mu_\rho)$. It now follows from Proposition~\ref{prop:ExtTrace} and~\eqref{eq:Energy-isom} that $u_f\in D^{1,p}(\Om, d, \mu)$ with $Tu_f=f$ on $\partial\Om$ and $u_f$ is $p$-harmonic on $(\Om,d,\mu)$.

Finally, suppose that $v$ is another solution in $(\Om,d,\mu)$ with boundary data $f$. Then by the uniform convexity of $L^p(\Om,\mu)$, we know that $g_v=g_{u_f}$ $\mu$-a.e.~in $\Om$, see for instance. 
Thus $g_{u_f}$ is also the minimal $p$-weak upper gradient of both $\min\{u_f,v\}$ and $\max\{u_f,v\}$, see \cite[Corollary 2.20]{bjornbjorn}.
Let $a,b \in \Q$ with $a<b$ and consider the set $W[a,b]:=\{x\in \Om\, :\, v(x)\ge b>a\ge u_f(x)\}$.
We fix a real number $c$ with $a<c<b$, and set a function $w$ on $\Om$ by 
\[
w=\min\{v,\, \max\{u_f,c\}\}.
\]
Note that then $w(x)=c$ whenever $x\in W[a,b]$.
It follows that $w\in D^{1,p}(\Om,d,\mu)$ with $g_w\le g_{u_f}=g_v$ $\mu$-a.e. on $\Om$ while $g_w=0$ $\mu$-a.e. on $W[a,b]$.
Furthermore, $T w=f$.
This is because $v=f=u_f$ $\nu$-a.e. on $\partial \Om$ and also $T w(x)=w(x)$ for $\nu$-a.e. $x\in\partial \Om$. 
Note that if $\cp_p(A)=0$ for $A \subset \partial \Omega$, then $\nu(A)=0$, see Lemma~\ref{lem:capacity2nu} above.
It follows that we have
\[
\Delta=\int_\Om g_{u_f}^p\, d\mu
\leq \int_\Om g_w^p\, d\mu
\le \int_{\Om\setminus W}g_{u_f}^p\, d\mu,
\]
which is only possible if $g_{u_f}=0$ $\mu$-a.e.~in $W$.
Since $U_+:=\{x\in\Om\, :\, v(x)>u_f(x)\}$ is the countable union of sets of form $W[a,b]$, it follows that $g_{u_f}=0$ $\mu$-a.e.~on $U_+$.
Reversing the roles of $u_f$ and $v$ in the above argument, we also get that $g_{u_f}=0$ $\mu$-a.e.~on the set $\{x\in\Om\, :\, v(x)<u_f(x)\}$. Thus $0\le g_{u_f-v}\le g_{u_f}+g_v=0$ $\mu$-a.e.~on the set $U:=\{x\in\Om\, :\, u_f(x)\ne v(x)\}$. On the set $\Om\setminus U$ we have $u_f=v$, and so $g_{u_f-v}=0$ $\mu$-a.e.~on $\Om\setminus U$ as well, that is, $g_{u_f-v}=0$ $\mu$-a.e.~on $\Om$. As $\Om$ is connected and supports a $p$-Poincar\'e inequality, and as $u_f$ and $v$ are continuous on $\Om$, it follows that $u_f-v$ is constant on $\Om$. As $T(u_f-v)=0$ on $\partial\Om$, it follows that $u_f-v=0$ on $\Om$ completing the proof.
\end{proof}

\section{Examples} \label{Sec:Examples}

In this section we give three examples to illustrate the techniques developed in the prior sections.

\begin{example}\label{Example-0}
We begin with an example most familiar to the reader. With $n\ge 2$ a positive integer, we consider $\Om$ to be the upper half-space $\Om:=\{(x_1,\cdots, x_n)\in \R^n\, :\, x_n>0\}$, equipped with the Euclidean metric and the standard $n$-dimensional Lebesgue measure.
Such $\Om$ satisfies all of our assumptions related to $\Om$, and $\partial\Om$ is isometrically $\R^{n-1}$. The choice of $b=(0,\cdots, 0)$ and
\[
\rho(t)=\frac{1}{(1+t)^2}
\]
yields the transformed space $(\Om, d_\rho)$ as the classical stereographic projection of $\Om$ to the open spherical cap in $n$-dimensional sphere in $\mathbb{R}^{n+1}$, with boundary $\partial_\rho\Om$ as an $(n-1)$-dimensional  great circle. This choice of $\rho$ satisfies our conditions \ref{condA}--\ref{condD} if $p>(n+1)/2$. Note here that $\theta=1$. 
\end{example}

\begin{example}\label{Example-1}
In this example we set $\rho(t)=\min\{1,t^{-\beta}\}$ with $\beta$ large enough.
Suppose that our $\Omega$, $\mu$ and $\nu$ are as stated in the standing assumptions and $b\in\partial\Omega$ is a fixed point. Then $\rho(t)=\min\{1,t^{-\beta}\}$, for sufficiently large $\beta$, satisfies the conditions \ref{condA}, \ref{condB}, \ref{condC} and \ref{condD}.
The conditions \ref{condA} and \ref{condB} follow by direct calculation needing $\beta>1$. We can take $C_A=3^\beta$ and $C_B=\beta/(\beta-1)$.
Recall that in light of condition~\ref{condA} and the doubling properties of $\mu$ and $\nu$, the conditions~\ref{condC} and~\ref{condD} need to be verified only for the case $r \ge 1$.
Thus suppose that $r\ge 1$. Then
\begin{align*}
    \int_{\Omega\setminus B(b,r)}\rho(|x|)^p\,d\mu(x)&=\sum_{j=0}^\infty\int_{B(b,2^{j+1}r)\setminus B(b,2^jr)}\rho(|x|)^p\,d\mu(x)\\
    &\le\sum_{j=0}^\infty\rho(2^jr)^p \mu(B(b,2^{j+1}r))\\
    &\le C_\mu\rho(r)^p\mu(B(b,r))\sum_{j=0}^\infty\left(\frac{C_\mu}{2^{\beta p}}\right)^j,
\end{align*}
where the first inequality comes from monotonicity of $\rho$ and $\mu$, and the second uses the doubling property of $\mu$ and the definition of $\rho$. The condition \ref{condC} follows, if the sum is converging, and this happens, if $\beta>\log_2(C_\mu)/p$.
Since we know that $\nu$ is also a doubling measure, the condition \ref{condD} is showed similarly, just $p$ is replaced with $p-\theta$ and $\mu$ with $\nu$. For condition \ref{condD} we need $\beta>\log_2(C_\nu)/(p-\theta)$. 
Recall that we can choose $C_{\nu} \leq 2^{-\theta} C_{\mu} C_{\theta}^2$, see \eqref{eq:codim-def}.

In the case that $\Om$ is the upper half-space in $\R^n$ as in Example~\ref{Example-0} above, we have $C_\mu=2^n$ and $C_\nu=2^{n-1}$, and so we would need $\beta>\max\{1, (n-1)/(p-1)\}$.
\end{example}

\begin{example}\label{Example-2}
In \cite{bjornbjornkorterogovintakala} it was shown that $B^{1-\theta/p}_{p,p}$-Besov energy is preserved under sphericalization when using the metric density function $\rho\colon[0,\infty)\to(0,\infty)$ defined as
\begin{equation}\label{eqn:bjornrho}
    \rho(t)=\frac{1}{(t+1)\nu( \overline{B}(b,t+1))^{1/(p-\theta)}},
\end{equation}
where $\overline{B}(b,t+1))$ means the closed ball of radius $t+1$. 
Actually, an open ball $B(b,t+1)$ was used instead of the closed ball $\overline{B}(b,t+1)$ in \cite{bjornbjornkorterogovintakala}. 
Because of the doubling property of the measure $\nu$, the above change to the closed ball changes only some numerical constants in \cite{bjornbjornkorterogovintakala}, but otherwise everything goes through. The reason why we take closed ball here is that we would like our $\rho$ to be lower semicontinuous and with closed balls it is.

The sphericalized metric in \cite{bjornbjornkorterogovintakala}, denoted by $\hat d$, was defined via taking infimum of sums over chains instead of integral over curves, because the space where the Besov space is naturally considered is often just a boundary of a domain and thus there are not necessarily any rectifiable curves.  

To compare the sphericalization in \cite{bjornbjornkorterogovintakala} to the current work, consider the setting of the present paper. Then $\partial\Om$, equipped with the metric $d$ and measure $\nu$, satisfies all the assumptions of the setting considered in~\cite{bjornbjornkorterogovintakala}.
The function $\rho$ defined in \eqref{eqn:bjornrho} satisfies conditions \ref{condA}, \ref{condB}, \ref{condC} and \ref{condD}.

The condition \ref{condA} follows directly using the fact that $\nu$ is also a doubling measure. The condition \ref{condB} comes by using the upper mass bound for $\nu$, see \eqref{eqn:massbounds}. For condition \ref{condC} let $r\ge 1$. Then
\begin{align*}
\int_{\Omega\setminus B(b,r)}\rho(|x|)^p\,d\mu(x)
&=\sum_{j=0}^\infty\int_{B(b,2^{j+1}r)\setminus B(b,2^jr)}\rho(|x|)^p\,d\mu(x)
\le\sum_{j=0}^\infty\rho(2^jr)^p \mu(B(b,2^{j+1}r))
\\
&\le C_\mu\sum_{j=0}^\infty\frac{\mu(B(b,2^jr))}{(2^jr)^p\nu(B(b,2^jr))^{p/(p-\theta)}}
\le C_\theta C_\mu r^{\theta-p}\sum_{j=0}^\infty\frac{(2^{(\theta-p)})^j} {\nu(B(b,2^jr))^{\theta/(p-\theta)}}
\\
&\leq \frac{C_{\theta} C_{\mu}}{1-2^{(\theta -p)}}\, \frac{r^{\theta-p}}{\nu(B(b,r))^{\theta/(p-\theta)}}
\leq \frac{C_{\theta}^2 C_{\mu}}{1-2^{(\theta -p)}}\, \frac{\mu(B(b,r))}{r^p\nu(B(b,r))^{p/(p-\theta)}}
\\
&\simle \rho(r)^p\mu(B(b,r+1)),
\end{align*}
where for the third and fifth inequality the codimensionality of $\nu$ was used.
The calculations for the condition~\ref{condD} are similar, but easier.

Thus the function $\rho$ presented in \cite{bjornbjornkorterogovintakala} is suitable for all the results in~\cite{KRST} and this current paper.
Therefore  $B^{1-\theta/p}_{p,p}$-Besov energy is preserved via this article's results also. However, the sphericalized metric $\hat d$  defined in \cite{bjornbjornkorterogovintakala} was slightly different, but it is bi-Lipschitz equivalent to the metric $d_\rho$  that is used in the current paper and also in \cite{KRST}. The bi-Lipschitz equivalence can be seen by comparing Lemmas~\ref{lemma4.1} and \ref{lemma4.2} with \cite[Lemma 4.2~(b) and Corollary 5.4]{bjornbjornkorterogovintakala}.
\end{example}

\section{\texorpdfstring{$p$}{}-hyperbolicity and \texorpdfstring{$p$}{}-parabolicity}
\label{sect-hyperbolicity}

In this section our goal is to prove Theorem~\ref{thm:Main2}, dealing with the notions of $p$-parabolicity and $p$-hyperbolicity.
The notions of $p$-parabolicity and $p$-hyperbolicity of a metric space have their roots in the theory of Brownian motion, where a topological measure space, equipped with a Dirichlet form, is said to be recurrent if Brownian motion on the space returns to any compact set of positive capacity in the space, and transient if it is not recurrent. 
The transience of Brownian motion was equated with the existence of singular functions on the space~\cite{Grigoryan} and with the positivity of the $2$-capacity of the boundary at infinity, relative to a compact subset of the space~\cite{Holo, HoloKos, HoloSh}. 
The terminology of transience was renamed as $2$-hyperbolicity~\cite{Holo}. The corresponding notion of $p$-hyperbolicity as positivity of the $p$-capacity of the boundary at infinity relative to compact subsets of the space makes sense even for $p\ne 2$ (see Definition~\ref{def:parab} below), and was first proposed by Holopainen in the context of Hadamard manifolds~\cite{Holo}.

\begin{definition}\label{def:parab}
A metric measure space $(Y,d_Y,\mu_Y)$ is said to be \emph{$p$-hyperbolic} if there is a ball $B\subset Y$ such that with $\Gamma(\infty)$ the collection of all locally rectifiable curves $\gamma:[0,\infty)\to Y$, with $\gamma(0) \in \overline{B}$ such that for each $n\in\N$ there is some $t_n>0$ such that $\gamma([t_n,\infty))\subset Y\setminus nB$, satisfies $\Mod_p(\Gamma(\infty))>0$.
We say that $(Y,d_Y,\mu_Y)$ is \emph{$p$-parabolic} if it is not $p$-hyperbolic.
\end{definition}

Suppose that the space $(Y,d_Y)$ is complete, $\mu_Y$ is doubling, and $(Y, d_Y,\mu_Y)$ supports a $p$-Poincar\'e inequality. Then
if $(Y,d_Y,\mu_Y)$ is $p$-hyperbolic, the positive $p$-modulus property in the definition of $p$-hyperbolicity holds for every ball $B$ in $Y$.
This follows from~\cite[Remark 3.10 and Lemma~3.15]{HoloSh}.

From Lemma~\ref{lem:zero} we know that if $\rho$ satisfies Conditions~\ref{condA}--\ref{condD} and $u\in D^{1,p}(\Om,d,\mu)$ such that $Tu=0$, then the collection of all points $\xi\in\partial_\rho\Om$ for which $T_\rho u(\xi)\ne 0$, is of null $p$-capacity with respect to the Sobolev class $N^{1,p}(\overline{\Om}^{d_\rho},d_\rho,\mu_\rho)$.
This means that if $(\overline{\Omega}^d ,d,\mu)$ is $p$-hyperbolic, then $T_\rho u(\infty)$ must equal $0$.

In the rest of this section we are not concerned with traces of functions in $D^{1,p}(\Om, d,\mu)$, and so we don't need to assume condition~\ref{condD}.
We also don't need to assume that $\partial \Omega$ is equipped with a $\theta$-codimensional measure $\nu$, or that $(\partial \Omega,d)$ is uniformly perfect or that it is unbounded. However we do still assume that $\Omega$ is an unbounded uniform domain in the space $(X,d)$.
In particular, the results of this section apply to the setting where $\Om$ is an unbounded uniform domain equipped with a measure that is doubling and supporting a $p$-Poincar\'e inequality, even if $\partial\Om$ is bounded as in~\cite{GibaraKorteSh}.

We now obtain the key result that allows us to use the tool of sphericalization to prove Theorem~\ref{thm:Main2}.

\begin{proposition}
\label{lem-parabolic-capacity}
Let $p > 1$.
The space $(\overline \Omega^d,d,\mu)$ is $p$-hyperbolic if and only if
\begin{equation*}
\rcapa_p(\{\infty\}, \overline \Omega^{d_{\rho}} \setminus B_\rho(\infty,r_0); \overline \Omega^{d_{\rho}})>0
\end{equation*}
for some $r_0>0$.
Here the variational $p$-capacity is with respect to the space $(\overline \Omega^{d_{\rho}},d_{\rho},\mu_{\rho})$.
\end{proposition}

\begin{proof}
Let $B$ be a ball in $(\overline \Omega^d,d)$ and set $\Gamma(\infty)$ to be the collection of all locally rectifiable curves $\gamma:[0,\infty)\to\overline{\Om}^d$ with $\gamma(0)\in \overline{B}$ and for each positive integer $n$ there is some $t_n>0$ such that $\gamma([t_n,\infty))\subset \overline{\Om}^d\setminus nB$. We also set $\Gamma_\rho(\infty)$ to be the collection of all rectifiable $\gamma:[0,1]\to\overline{\Om}^{d_\rho}$ with $\gamma(0)\in \overline{B}$ such that $\gamma(1)=\infty$, and set $\Gamma_\rho^P(\infty)$ to be the collection of all curves $\gamma\in\Gamma_\rho(\infty)$ for which $\gamma([0,1))\subset\overline{\Om}^d$.
Then
\[
\Mod_{\rho,p}(\Gamma_\rho(\infty))=\Mod_{\rho,p}(\Gamma_\rho^P(\infty)),
\]
where $\Mod_{\rho,p}$ denotes the $p$-modulus $\Mod_p$ of the curve family measured in the space $(\overline{\Om}^{d_\rho}, d_\rho,\mu_\rho)$.

Since $\diam_\rho(\overline{\Om}^{d_{\rho}} \setminus B(b,R))\to 0$ as $R\to\infty$ by \eqref{balls-at-infinity}, it follows that each path $\gamma\in\Gamma(\infty)$ satisfies $\lim_{t\to\infty}d_\rho(\gamma(t),\infty)=0$. By re-parametrizing $\gamma$ via the transformation $\phi:[0,1)\to[0,\infty)$ given by $\phi(t)=\tfrac{t}{1-t}$, we obtain $\beta=\gamma\circ \phi:[0,1)\to\overline{\Om}^d$; then extending $\beta$ to $[0,1]$ by setting $\beta(1)=\infty$, we obtain a one-to-one correspondence between $\Gamma(\infty)$ and $\Gamma_\rho^P(\infty)$.
Note that a nonnegative Borel function $g$ on $\overline{\Om}^d$ is admissible for computing $\Mod_p(\Gamma(\infty))$ if and only if the function $\widehat{g}$ given by $\widehat{g}(x)=\rho(|x|)^{-1}g(x)$ is admissible for computing $\Mod_{\rho,p}(\Gamma_\rho^P(\infty))$; this follows from the same argument found in the proof of~\cite[Lemma 5.3]{KRST}.
Because we also have $\int_{\Omega} g^p d \mu = \int_{\Omega} \widehat{g}^p d \mu_{\rho}$, we obtain
\[
\Mod_{\rho,p}(\Gamma_\rho^P(\infty))=\Mod_p(\Gamma(\infty)).
\]
Combining the above, we see that
\[
\Mod_{\rho,p}(\Gamma_\rho(\infty))=\Mod_p(\Gamma(\infty)).
\]
An appeal to~\cite[Remark~3.3]{KalSh} 
shows that 
\[
\rcapa_p(\{\infty\},\overline{B};\overline{\Om}^{d_\rho})=\Mod_{\rho,p}(\Gamma_\rho(\infty)),
\]
and combining this with the previous equation and \eqref{balls-at-infinity} yields the desired result.
\end{proof}

The results in~\cite{HoloKos} tell us of a volume growth property that guarantees $p$-parabolicity of a manifold, while the results in~\cite{Korte} give conditions under which a point in a metric measure space has positive $p$-capacity (or not).
The above proposition tells us that~\cite{Korte} gives a way of determining $p$-parabolicity and $p$-hyperbolicity of $\overline{\Om}^d$.
In this section we give two volume growth properties, 
see Theorems~\ref{thm:p-parabolic} and \ref{thm:hypo} below.
To prove these, we use the sphericalization tool.
We also provide properties that guarantee the positivity (or nullness) of $p$-capacity of $\{\infty\}$ along the lines of~\cite[Theorem~3.4]{Korte}. These properties are described in the corresponding propositions stated below.

\begin{proposition}\label{prop:parab}
Let $p>1$ and let $(Y,d_Y,\mu_Y)$ be a metric measure space such that $\mu_Y$ is doubling.
Fix $x_0\in Y$ and $r_0>0$. If
\[
\liminf_{r\rightarrow0^+}\frac{\mu_Y(B(x_0,r))}{r^p}<\infty
\]
then $\rcapa_p(\{x_0\}, Y\setminus B(x_0,r_0);Y)=0$.
\end{proposition}

\begin{proof}
Fix $M> \liminf_{r\rightarrow0^+}\frac{\mu_Y(B(x_0,r))}{r^p}$.
Then there exists a sequence $\{r_i\}_{i\in\N}$ such that $r_1\le r_0$, $0<r_{i+1}<\tfrac{r_i}{2}$ and $\mu_Y(B(x_0,r_i))/r_i^p\le M$ for every $i\in\N$. For each $i\in\N$, let
\[
u_i(x)= \left(1-\frac{\dist(x,B(x_0,r_i))}{r_i}\right)_+.
\]
Then $u_i=1$ in $B(x_0,r_i)$, $u_i=0$ in $Y\setminus B(x_0,2r_i)$, $g_{u_i}\le 1/r_i$ in $B(x_0,2r_i)\setminus B(x_0,r_i)$ and 
\[
\int_Y g_{u_i}^p d\mu_Y
\leq \frac{\mu_Y(B(x_0,2r_i))}{r_i^p}
\le C_{\mu_Y} M.
\]
Notice that the functions $g_{u_i}$ have pairwise disjoint supports.
Now fix a positive integer $K$. Let
\[
v_K=\frac1K\sum_{i=1}^K u_i.
\]
We notice that $v_K=0$ in $Y\setminus B(x_0,r_0)$ and that $v_K=1$ in $B(x_0,r_K)$. Consequently
\[
\rcapa_p(\{x_0\},Y\setminus B(x_0,r_0);Y)
\le \int_Y g_{v_K}^p d\mu_Y
=\sum_{i=1}^K \int_Y \frac{1}{K^p}g_{u_i}^p d\mu_Y
\le K^{1-p} C_{\mu_Y} M.
\]
By letting $K\rightarrow \infty$ this completes the proof.
\end{proof}

\begin{theorem}\label{thm:p-parabolic}
Let $p>1$ and suppose that 
\[
\liminf_{R\to\infty}\frac{\mu(B(b,R))}{R^p}<\infty.
\]
Then $(\overline{\Om}^d,d,\mu)$ is $p$-parabolic.
\end{theorem}

\begin{proof}
Let $R_i$ be a sequence such that $\lim_{i \to \infty} R_i = \infty$ and $\sup_i\frac{\mu(B(b,R_i))}{R_i^p}<\infty$.
Recall from \cite{KRST} or the beginning of Section~\ref{sect-codim} that $h(t) = (t+1)\rho(t)$.
By setting $r_i = h(R_i)$, we have $\lim_{i \to \infty} r_i = 0$, see \cite[Remark 2.1]{KRST}.
If $i$ is large enough, we get from \eqref{balls-at-infinity} and Condition \ref{condC} that
\begin{equation*}
\mu_{\rho}(B_{\rho}(\infty,r_i))
\leq \mu_{\rho}(\Omega \setminus B(b,h^{-1}(r_i/C_1)))
\leq C_C \rho(h^{-1}(r_i/C_1))^p \mu(B(b,h^{-1}(r_i/C_1)+1)).
\end{equation*}
For large $i$, we get from \cite[Lemma 2.2]{KRST} that
\begin{equation*}
R_i\ge h^{-1}(h(R_i)) \geq \frac{R_i}{2(C_A C_B)^{C_A C_B}}.
\end{equation*}
Therefore we have $h^{-1}(h(R_i)) \simeq R_i$. Then by using \cite[Lemma 4.2]{KRST}, condition \ref{condA} and the doubling property of $\mu$, we get
\begin{equation*}
\rho(h^{-1}(r_i/C_1))^p \mu(B(b,h^{-1}(r_i/C_1)+1))
\simeq \rho(h^{-1}(r_i))^p \mu(B(b,h^{-1}(r_i)+1))
\simeq \rho(R_i)^p \mu(B(b,R_i)),
\end{equation*}
where the comparison constants depend only on $C_U$, $C_{\mu}$, $p$, $C_A$ and $C_B$.
Therefore
\begin{equation*}
\limsup_{i \to \infty} \frac{\mu_\rho(B_\rho(\infty,r_i))}{r_i^p}
\simle \limsup_{i \to \infty} \frac{\rho(R_i)^p \mu(B(b,R_i))}{(R_i \rho(R_i))^p}
< \infty.
\end{equation*}
Thus Proposition~\ref{prop:parab} tells us that 
\[
\rcapa_p(\{\infty\}, \overline \Omega^{d_{\rho}} \setminus B_\rho(\infty,r_0); \overline \Omega^{d_{\rho}})=0
\]
whenever $r_0>0$. Now an application of Proposition~\ref{lem-parabolic-capacity} completes the proof.
\end{proof}

\begin{proposition}\label{prop:positive}
Let $(Y,d_Y,\mu_Y)$ be a complete metric measure space such that $\mu_Y$ is doubling and the space supports a $p$-Poincar\'e inequality.
Fix $x_0\in Y$.
Suppose that there is some $q>0$ such that $0<q<p<\infty$ and 
\[
\liminf_{r\to 0^+}\frac{\mu_Y(B(x_0,r))}{r^q}>0.
\]
Then for each $r_0>0$ with $Y\setminus B(x_0,2r_0)$ non-empty, we have 
\[
\rcapa_p(\{x_0\}, Y\setminus B(x_0,r_0); Y)>0.
\]
\end{proposition}

\begin{proof}
Let $0<r<r_0$ such that $Y\setminus B(x_0,2r_0)$ is non-empty, and $u\in N^{1,p}(Y,d_Y,\mu_Y)$ such that $u=1$ on $B(x_0,r)$, $u=0$ on $Y\setminus B(x_0,r_0)$ and $0\le u\le 1$ on $Y$.
Let $k \geq 2$ be the integer for which $2^{-k+1}r_0<r\le 2^{-k+2} r_0$, and for $i=0, 1,\cdots, k$ set $B_i:=B(x_0,2^{1-i}r_0)$.

Since $Y\setminus B(x_0,2r_0)$ is non-empty and $Y$ is connected (as a consequence of supporting a $p$-Poincar\'e inequality), it follows that there is some $z_0\in Y$ with $d_Y(x_0,z_0)=\tfrac{3r_0}{2}$. Therefore, by the doubling property of $\mu_Y$, 
\[
\mu_Y(B(x_0,2r_0)\setminus B(x_0,r_0))\ge \mu_Y(B(z_0,r_0/2))\ge \frac{1}{C_{\mu_Y}^3}\, \mu_Y(B(x_0,r_0)),
\]
and so
\[
\mu_Y(B(x_0,2r_0))\ge \left(1+\frac{1}{C_{\mu_Y}^3}\right)\, \mu_Y(B(x_0,r_0)).
\]

Since $u=0$ on $Y\setminus B(x_0,r_0)$ and $0\le u\le 1$ on $B(x_0,r_0)$, it follows that
\[
u_{B_0}=\frac{1}{\mu_Y(B(x_0,2r_0))}\, \int_{B(x_0,2r_0)}u\, d\mu_Y\le \frac{\mu_Y(B(x_0,r_0))}{\mu_Y(B(x_0,2r_0))}\le c<1,
\]
where $c=(1+\tfrac1C_{\mu_Y}^3)^{-1}$.
Therefore
\begin{align*}
0<1-c
\leq 1-u_{B_0}
= |u_{B_0}-u_{B_k}|
\le \sum_{j=0}^{k-1}|u_{B_j}-u_{B_{j+1}}|
&\le C_{\mu_Y} \sum_{j=0}^{k-1}\vint_{B_j}|u-u_{B_j}|\, d\mu_Y
\\
&\le C_{\mu_Y} C_{P_Y} \sum_{j=0}^{k-1}2^{1-j}r_0\, \left(\vint_{\lambda B_j}g_u^p\, d\mu_Y\right)^{1/p}.
\end{align*}
It follows that
\[
1-c\le 2 C_{\mu_Y} C_{P_Y} \sum_{j=0}^\infty \left(\frac{(2^{-j}r_0)^q\, (2^{-j}r_0)^{p-q}}{\mu_Y(\lambda B_j)}\right)^{1/p}\, 
\left(\int_Yg_u^p\, d\mu_Y\right)^{1/p}.
\]

By the hypothesis, there exist a positive integer $k_1$ and a constant $c_1>0$ so that whenever $j>k_1$, we have $\mu_Y(\lambda B_j)\ge c_1\, (2^{-j}r_0)^q$.
Hence
\[
1-c\le 2 C_{\mu_Y} C_{P_Y} \left(\int_Yg_u^p\, d\mu_Y\right)^{1/p}\, \bigg[\sum_{j=0}^{k_1} \frac{2^{-j}r_0}{\mu_Y(\lambda B_j)^{1/p}}
+\sum_{j=k_1+1}^\infty \left(\frac{r_0^{p-q}}{c_1}\right)^{1/p}\, 2^{-j(1-q/p)}\bigg].
\]
The term within the square parenthesis above is now independent of $u$ and $r$, and the series $\sum_{j=1}^\infty 2^{-j(1-q/p)}$ is finite. Hence, setting
\[
0<A:=2 C_{\mu_Y} C_{P_Y} \bigg[\sum_{j=0}^{k_1} \frac{2^{-j}r_0}{\mu_Y(\lambda B_j)^{1/p}}
+\sum_{j=k_1+1}^\infty \left(\frac{r_0^{p-q}}{c_1}\right)^{1/p}\, 2^{-j(1-q/p)}\bigg]<\infty,
\]
we see that 
\[
\left(\int_Yg_u^p\, d\mu_Y \right)^{1/p}\ge \frac{1-c}{A}>0.
\]
Taking the infimum over all $u$ and then letting $r\to 0^+$ yields the claim of the proposition thanks to~\cite[Theorem 6.19]{bjornbjorn}.
\end{proof}

\begin{theorem}\label{thm:hypo}
Let $p>1$.
Suppose that there is some $q_0>p$ such that 
\[
\liminf_{R\to\infty}\frac{\mu(B(b,R))}{R^{q_0}}>0.
\]
Then $(\overline{\Om}^d,d,\mu)$ is $p$-hyperbolic.
\end{theorem}

\begin{proof}
Let $\beta>\max\left\{\frac{q_0}{p} , \frac{\log_2(C_{\mu})}{p} \right\}$.
From Example~\ref{Example-1} we know that the choice of $\rho$ given by $\rho(t)=\min\{1,t^{-\beta}\}$ satisfies conditions~\ref{condA}, \ref{condB} and \ref{condC}.
Recall that condition~\ref{condD} is not needed in this section.
Let $q=\tfrac{\beta p-q_0}{\beta -1}$. Note that as $p<q_0$, necessarily $0<q<p$.

Let $(r_i)_{i=1}^{\infty}$ be a sequence of positive numbers such that $\lim_{i \to \infty} r_i = 0$ and $r_i \leq 1$ for every $i$.
Recall from the beginning of Section \ref{sect-codim} that $h(t) = (t+1)\rho(t)$.
Define $R_i = h^{-1}(r_i)$. Then $R_i \geq 1$ and $h(R_i) = (R_i+1)R_i^{-\beta} = r_i$ for every $i$, and $\lim_{i \to \infty} R_i = \infty$.
If $i$ is large enough, we get from \eqref{balls-at-infinity}, \eqref{condCreverse} and \cite[Lemma 4.2]{KRST} that
\begin{align*}
\mu_{\rho}(B_{\rho}(\infty,r_i))
&\geq \int_{\Omega \setminus B \left( b,2 h^{-1} \left( \frac{r_i}{2 C_2 C_A C_B} \right) \right) } \rho(|x|)^p d \mu(x)
\\
&\simge \rho \left( 2 h^{-1} \left( \frac{r_i}{2 C_2 C_A C_B} \right) \right)^p \mu \left( B \left( b,2 h^{-1} \left( \frac{r_i}{2 C_2 C_A C_B} \right)+1 \right) \right)
\\
&\simge \rho(h^{-1}(r_i))^p \mu(B(b,h^{-1}(r_i)))
\end{align*}
with comparison constants depending on $C_U$, $C_{\mu}$, $p$ and $\beta$.
It follows that
\begin{equation*}
\frac{\mu_\rho(B_\rho(\infty, r_i))}{r_i^q}
\simge \frac{\rho(R_i)^p \mu(B(b,R_i))}{ \left( \left( R_i+1 \right) R_i^{-\beta} \right)^q}
\simeq \frac{R_i^{-\beta p} \mu(B(b,R_i))}{R_i^{(1-\beta)q}}
= \frac{\mu(B(b,R_i))}{R_i^{q_0}}
\end{equation*}
with comparison constants also depending on $q$.
Therefore the hypothesis of Proposition \ref{prop:positive} is satisfied with $(Y,d_Y,\mu_Y)=(\overline{\Om}^{d_{\rho}}, d_\rho,\mu_\rho)$ and $x_0 = \infty$.
Thus if $r_0>0$ is small enough, we have
\[
\rcapa_p(\{\infty\}, \overline \Omega^{d_{\rho}} \setminus B_\rho(\infty,r_0); \overline \Omega^{d_{\rho}})>0.
\]
An application of Proposition~\ref{lem-parabolic-capacity} now completes the proof.
\end{proof}

	\noindent {\bf Addresses:}
	
\vskip .5cm
		
	\noindent R.K.: Department of Mathematics and Systems Analysis, Aalto University, P.O. Box 11100, FI-00076 Aalto, Finland.
	\\
	\noindent E-mail:  R.K.: {\tt riikka.korte@aalto.fi}
		
\vskip .5cm

         \noindent S.R.: Department of Mathematics and Systems Analysis, Aalto University, P.O. Box 11100, FI-00076 Aalto, Finland.
	\\
	\noindent E-mail:  S.R.: {\tt sari.rogovin@aalto.fi}
		
\vskip .5cm
        
	\noindent N.S.: Department of Mathematical Sciences, P.O.~Box 210025, University of Cincinnati, Cincinnati, OH~45221-0025, U.S.A.
    \\
	\noindent E-mail:  N.S.: {\tt shanmun@uc.edu}

\vskip .5cm
		
	\noindent T.T.: Department of Mathematics and Systems Analysis, Aalto University, P.O. Box 11100, FI-00076 Aalto, Finland
    \\
	\noindent E-mail:  T.T.: {\tt timo.i.takala@aalto.fi}

\end{document}